\documentclass[11pt]{amsart}
\usepackage[numbers, square]{natbib}
\usepackage{amssymb}
\usepackage{amsmath}
\usepackage{amsfonts}
\usepackage{geometry}
\usepackage{graphicx}
\usepackage{amsthm}
\usepackage{hyperref}
\usepackage[dvipsnames]{xcolor}

\providecommand{\U}[1]{\protect\rule{.1in}{.1in}}
\theoremstyle{plain}

\newtheorem{lemma}{Lemma}

\newtheorem{proposition}{Proposition}
\newtheorem{remark}{Remark}

\newtheorem{theorem}{Theorem}
\numberwithin{equation}{section}
\newcommand{\lap}{\mbox{$\triangle$}}

\newcommand{\disp}{\displaystyle}

\newcommand{\eps}{\varepsilon}

\newcommand{\be}{\beta}

\newcommand{\te}{\theta}
\newcommand{\la}{\lambda}

\newcommand{\iny}{\infty}

\newcommand{\su}{\subset}
\newcommand{\LP}{\Delta}
\newcommand{\gr}{\nabla}

\newcommand{\norm}[1]{\left\vert\left\vert #1\right\vert\right\vert}

\newcommand{\abs}[1]{\left\vert#1\right\vert}
\newcommand{\set}[1]{\left\{#1\right\}}
\newcommand{\brac}[1]{\left[#1\right]}
\newcommand{\pr}[1]{\left( #1 \right) }
\newcommand{\pb}[1]{\left( #1 \right] }
\newcommand{\bp}[1]{\left[ #1 \right) }

\newcommand{\R}{\ensuremath{\mathbb{R}}}

\newcommand{\C}{\ensuremath{\mathbb{C}}}

\begin{document}
\title[Quantitative unique continuation ]
{Quantitative unique continuation of solutions to higher order elliptic
equations with singular coefficients}
\author{ Jiuyi Zhu}
\address{
Department of Mathematics\\
Louisiana State University\\
Baton Rouge, LA 70803, USA\\
Email:  zhu@math.lsu.edu }
\thanks{
Zhu is supported in part by  NSF grant DMS-1656845}
\date{}

\begin{abstract}
 We investigate the quantitative unique continuation of solutions
to higher order elliptic equations with singular coefficients. Quantitative unique continuation
described by the vanishing order is a quantitative form of strong
unique continuation property. We characterize the vanishing order of
solutions for higher order elliptic equations in terms of the norms
of coefficient functions in their respective
Lebesgue spaces. New versions of quantitative Carleman estimates are
established.
\end{abstract}

\maketitle
\section{Introduction}

In this paper, we study the quantitative unique continuation for higher order
elliptic equations with singular lower order terms. Suppose $u$ is a
non-trivial solution to
\begin{equation}
\triangle^m u+ \sum^{\alpha_0}_{|\alpha|=1} V_\alpha(x)\cdot
D^\alpha u+V_0(x) u=0 \quad \mbox{in} \ \mathbb B_{10}, \label{goal}
\end{equation}
 where $\mathbb B_{10}$ is a ball centered at origin with radius $10$ in $\mathbb R^n$
 with $n\geq 2$, The value $\alpha_0$ is a positive integer. If $m$ is a positive even
 integer, the value $\alpha_0\leq [\frac{3m}{2}]-1$.
 If $m$ is a positive odd  integer, the value $\alpha_0\leq [\frac{3m}{2}]
 $.
Assume that $V_\alpha(x)\in L^\infty(\mathbb B_{10})$ and $V_0(x)\in
L^s(\mathbb B_{10})$ for some positive constant $s$ to be specified
later. We also normalize the solutions $u$ in \eqref{goal} as
$\|u\|_{L^\infty(\mathbb B_{1})}\geq 1$ and $\|u\|_{L^\infty(\mathbb
B_{10})}\leq \hat{C}$.

Quantitative unique continuation described by the vanishing order
characterizes how much the solution vanishes. We say vanishing order
of solution at $x_0$ is $l$, if $l$ is the largest integer such that
$D^\alpha u(x_0) =0$ for all $|\alpha|< l$, where $\alpha$ is a
multi-index. It is a quantitative way to describe the strong unique
continuation property. So we also call it quantitative uniqueness. Strong unique continuation property states
that if a solution that vanishes of infinite order at a point
vanishes identically. We know that all zeros of nontrivial solutions
of second order linear equations on smooth compact Riemannian
manifolds are of finite order. Especially, for classic
eigenfunctions on a compact smooth Riemannian manifold
$\mathcal{M}$,
 $$-\lap_{g} \phi_\lambda=\lambda \phi_\lambda \quad \quad \mbox{in} \ \mathcal{M}.$$
Donnelly and Fefferman in \cite{DF88} showed that  the maximal
vanishing order of $\phi_\lambda$ is everywhere less than
$C\sqrt{\lambda}$, here $C$ only depends on the manifold
$\mathcal{M}$. The vanishing order of classical eigenfunction
$\phi_\lambda$ is sharp and its sharpness can be seen from spherical
harmonics if $\mathcal{M}$ is a sphere. If the strong unique
continuation property holds for the solutions and solutions do not
vanish of infinite order, the vanishing order of solutions depends
on the potential functions and coefficient functions appeared in the
equations. It is interesting to characterize the vanishing order by
the potential function $V_0(x)$ and  coefficient functions
$V_\alpha(x)$ in (\ref{goal}).

Recently, there has been much interest in investigating the
vanishing order of solutions for (\ref{goal}) in the case $m=1$,
i.e. the second order elliptic equation
\begin{equation} \triangle u+ V_1(x)\cdot \nabla u+V_0(x) u=0  \quad
\mbox{in} \ \mathbb B_{10}. \label{goal2}
\end{equation}
Kukavica in \cite{Ku98} studied the vanishing order of solutions for
Schr\"{o}dinger equation
\begin{equation}
\lap u+V_0(x)u=0. \label{schro}
\end{equation}
 If $V_0(x) \in W^{1, \infty}$, Kukavica  established that the upper bound of
vanishing order is less than $C(1+\|V_0\|_{W^{1, \infty}})$,
 where $\|V_0\|_{W^{1, \infty}}=\|V_0\|_{L^\infty}+\|\nabla
V_0\|_{L^\infty}$. This upper bound is not sharp, which can be seen
from Donnelly and Fefferman's work in the case $V_0(x)=\lambda$.
Recently, the sharp vanishing order for solutions of (\ref{schro})
is shown to be less than $C(1+\sqrt{\|V_0\|_{W^{1, \infty}}})$
independently by Bakri in \cite{Bak12} and Zhu in \cite{Zhu16} by
different methods, since the exponent of the norm of potential
function $V_0(x)$ matches the one for eigenfunctions in Donnelly and
Fefferman's work.

If $V_0(x)\in L^\infty$, Bourgain and Kenig \cite{BK05} considered a
similar problem for (\ref{schro}) motivated by their work on
Anderson localization for the Bernoulli model. Bourgain and Kenig
established that
\begin{equation}
\|u\|_{L^\infty (\mathbb B_r)}\geq c_1r^{c_2
(1+\|V_0\|_{L^\infty}^{\frac{2}{3}})} \quad \quad \mbox{as}\ r\to 0,
\label{like}
\end{equation}
 where $c_1,
c_2$ depend only on $n$, $\hat{C}$.  The estimate (\ref{like})
shows that the order of vanishing for solutions is less than
$C(1+\|V_0\|_{L^\infty}^{\frac{2}{3}}). $ Kenig in \cite{Ken07} also
pointed out that the exponent $\frac{2}{3}$ of $\|V_0\|_{L^\infty}$
is sharp for complex valued $V_0$ based on Meshkov's example in
\cite{Mes92}.

 Davey in \cite{Dav14} generalized the quantitative unique continuation result to
solutions to more general elliptic equations of the form $\LP u +
V_1 \cdot \gr u + V_0 u = \la u$, where $\la \in \C$, and $V_0$,
$V_1$ are complex-valued potential functions with pointwise decay at
infinity. Davey proved that the order of vanishing for such
solutions is less than
$C(1+\|V_1\|_{L^\infty}^{2}+\|V_0\|_{L^\infty}^{2/3})$. See also the
similar work in \cite{Bak13}. Furthermore, The results in
\cite{Dav14} were extended to  variable-coefficient operators by Lin
and Wang in \cite{LW14}.

Based on Donnelly and Fefferman's work on the vanishing order of
eigenfunctions, Kenig \cite{Ken07} asked if the order of vanishing
can be reduced to $C(1+\|V_0\|_{L^\infty}^{1/2})$ for real-valued
$u$ and $V_0$ for the solutions in (\ref{schro}). It is related to a
quantitative form of Landis' conjecture in the real-valued setting.
In the late 1960s, E.M.~Landis conjectured that the bounded solution
$u$ to $\LP u - V_0 u = 0$ in $\R^n$ is trivial if $\abs{u\pr{x}}
\lesssim \exp\pr{- c \abs{x}^{1+}}$, where $V_0$ is a bounded
function.  By assuming that the bounded real valued $V_0(x)$ to be
nonnegative, Landis' conjecture was answered in \cite{KSW15} in
$\mathbb R^2$.

It is known  that the strong unique continuation property holds for
second order elliptic equation (\ref{goal2}) with singular lower
terms satisfying the integrability condition, i.e.
$$  V_1\in L^{t} \ \mbox{with} \ t>n \quad \mbox{and} \quad  V_0\in L^{\frac{n}{2}}. $$
See e.g.  \cite{H85}, \cite{JK85}, \cite{S90}, \cite{Wol92},
\cite{KT01} for the literature about strong unique continuation
property, to just mention a few. Recently, interest has been shifted
to know how the singular lower order terms control the order of
vanishing of solutions. Kenig and Wang in \cite{KW15} studied the
quantitative uniqueness of solutions to second order elliptic
equations with a drift term $\LP u + V_1 \cdot \gr u = 0$ in
$\mathbb R^2$ using complex analytic tools. They established the
vanishing order estimates for solutions in the case that real-valued
$V_1 \in L^s\pr{\R^2}$ for some $s \in \bp{2, \iny}$.
 In \cite{KT16}, Klein and Tsang studied quantitative unique continuation for solutions to $\LP u + V_0 u = 0$ motivated by
 spectral projection of Schr\"odinger operators,
 where $V_0 \in L^s$ for some $s \ge n $. Their tools are Carleman estimates as that in \cite{BK05} and
 Sobolev imbedding arguments.
 It seems that their method can not be adapted to study elliptic equations with singular gradient potentials.
Very recently, by a new quantitative $L^p\to L^q$ Carleman estimates
for a range of $p$ and $q$ value, Davey and the author in
\cite{DZ17} were able to deal with (\ref{goal2}) with both singular
gradient potential $V_1$ and singular potential $V_0$ for $n\geq 3$.
Our results not only work for a larger range of singular potentials
and gradient potentials, but also improve the
 previous results on vanishing order of solutions. For $n=2$, Davey and the author in \cite{DaZ17} further explored the  $L^p\to L^q$ Carleman estimates developed in \cite{DZ17}.
 They were able to characterize
 vanishing order for all admissible singular potentials and gradient potentials, which provides a complete description of quantitative uniqueness for second order elliptic equations in $n=2$.

Higher order elliptic equations are important models in the study of
partial differential equations. We
assume throughout the paper that $m\geq 2$. A nature question is to study the
quantitative uniqueness of higher order elliptic equations. However, it is relatively less explored in the literature. The strong unique
continuation property has been well investigated for higher order
elliptic equations. See e.g. \cite{CG99}, \cite{L07}, \cite{CK10},
to just mention a few. In particular, this property has been shown
for singular potential $V_0$ and singular coefficient functions
$V_\alpha$ in \cite{L07}. The $\alpha$ value up to $[\frac{3m}{2}]$
for unique continuation was given by Protter \cite{P60}.
The vanishing order for higher order elliptic equations was
considered in \cite{Zhu16}. For the model
\begin{equation}
\lap ^m u+V_0u=0 \quad \mbox{in} \ \mathbb B_{10}, \label{zhu}
\end{equation}
it was shown in \cite{Zhu16} that the vanishing order of $u$ is less
than $C\|V_0\|_{L^\infty}$ for $n\geq 4m$ by a variant of frequency
function.  Lin, Nagayasu and Wang studied a different quantitative
uniqueness result for higher order elliptic equations in
\cite{LNW11},  where the vanishing order of solutions was not
explicitly provided in term of the potential function $V_0$ and the
coefficient functions $V_\alpha$.

 A priori, we assume that $u \in W^{m,2}_{loc}\pr{\mathbb B_{10}}$ is a
weak solution to \eqref{goal}. By regularity theory, it follows that
$u \in W_{loc}^{2m,2} \cap L^\iny_{loc}$.  In the paper, the
notation $\mathbb B_{r}\pr{x_0} \su \R^n$ is denoted as the ball of
radius $r$ centered at $x_0$. When the center is clear in the
context, we simply write $\mathbb B_{r}$. To fully discuss the vanishing order for higher order elliptic equations,
our result is stated as three cases in term of the relation of
$n$ and $m$.

\begin{theorem}
Let $u$ be a solution to \eqref{goal} in $\mathbb B_{10}$. \\
 I): In the
case of $n> 4m-2$, assume that $s\in (\frac{2n}{3m}, \ \infty]$.
Then the vanishing order of $u$ in  $\mathbb B_{1}$ is less than
$C(1+\sum_{|\alpha|=1}^{\alpha_0}\|V_\alpha\|_{L^\infty}^\mu+\|V_0\|_{L^s}^\nu)$.
That is, for any $x_0\in \mathbb B_1$ and every $r$ sufficiently
small,
\begin{align*}
\|u\|_{L^\iny(\mathbb B_{r}(x_0))}
 &\ge c r^{C(1+\sum_{|\alpha|=1}^{\alpha_0}\|V_\alpha\|_{L^\infty}^\mu+\|V_0\|_{L^s}^\nu)},
\end{align*}
where
$$ \mu=\frac{2}{3m-2|\alpha|} \quad \mbox{and} \quad \nu=\frac{2s}{3ms-2n}.           $$
and $c = c\pr{n, m, s, \hat C}$, $C = C\pr{n, m, s, \hat C}$.

II): In the case of $ n= 4m-2$, assume that $s\in
(\frac{4(2m-1)}{3m}, \, \infty]$.  Then for any sufficiently small
constant $\eps>0$, the vanishing order of $u$ in $\mathbb B_{1}$ is
less than
$C(1+\sum_{|\alpha|=1}^{\alpha_0}\|V_\alpha\|_{L^\infty}^\mu+\|V_0\|_{L^s}^{\tilde{\nu}})$.
That is, for any $x_0\in \mathbb B_1$ and every $r$ sufficiently
small,
\begin{align*}
\|u\|_{L^\iny(\mathbb B_{r}(x_0))}
 &\ge c
 r^{C(1+\sum_{|\alpha|=1}^{\alpha_0}\|V_\alpha\|_{L^\infty}^\mu+\|V_0\|_{L^s}^{\tilde{\nu}})},
\end{align*}
where
$$ \mu=\frac{2}{3m-2|\alpha|} \quad \mbox{and} \quad \tilde{\nu}= \frac{2s}{3ms-4(2m-1)-2(2m-1)(s-2)\epsilon},         $$
and $c = c\pr{n, m, s, \eps, \hat C}$, $C = C\pr{n,  m, s, \eps,
\hat C}$.

III): In the case $2\leq n<4m-2$, assume that $s\in
(\frac{4(2m-1)}{3m}, \, \infty]$. Then the
vanishing order of $u$ in $\mathbb B_{1}$ is less than
$C(1+\sum_{|\alpha|=1}^{\alpha_0}\|V_\alpha\|_{L^\infty}^\mu+\|V_0\|_{L^s}^{\bar{\nu}})$. That
is, for any $x_0\in \mathbb B_1$ and every $r$ sufficiently small,
\begin{align*}
\|u\|_{L^\iny(\mathbb B_{r}(x_0))}
 &\ge c r^{C(1+\sum_{|\alpha|=1}^{\alpha_0}\|V_\alpha\|_{L^\infty}^\mu+\|V_0\|_{L^s}^{\bar{\nu}})},
\end{align*}
where
$$ \mu=\frac{2}{3m-2|\alpha|} \quad \mbox{and} \quad \bar{\nu}= \frac{2s}{3ms-4(2m-1)},         $$
and $c = c\pr{n, m, s,  \hat C}$, $C = C\pr{n,  m, s,
\hat C}$.
 \label{thh}
\end{theorem}
Before we proceed, let us give some comments on Theorem \ref{thh}.
\begin{remark}
1. The vanishing order of solution is heavily relied on the Carleman estimates in Theorem \ref{CCarlpqp}, which is split into three cases.
To obtain the Carleman estimates for singular weights in suitable Lebesgue spaces,  Sobolev inequalities are used in the arguments. The application of Sobolev embedding implies those cases by the relation of $n$ and $m$.

2. In \cite{Zhu16}, the author developed a variant of frequency function to obtain the vanishing order less
than $C\|V_0\|_{L^\infty}$ for $n\geq 4m$ in (\ref{zhu}). In addition to the situation $n\geq 4m$, Theorem \ref{thh} provides the description of vanishing order for all cases.
Observe that the vanishing order of solution is less than $C\|V_0\|_{L^\infty}^{\frac{2}{3m}}$ if $V_\alpha=0$ and $s=\infty$.
Theorem \ref{thh} not only improves the vanishing order in the case of $s=\infty$  in \cite{Zhu16}, but also enables us to deal with singular potential $V_0$ and non-trivial coefficient function $V_\alpha$.

3. Because of the rich results for second order elliptic equation in the case of $m=1$ in \eqref{goal}, we assume $m\geq 2$ in the paper. However, the statement in Case I and II still applies to the case $m=1$. Observe that those conclusions match the sharp results by \cite{BK05} in the case of $s=\infty$ and $m=1$.

\end{remark}

Based on the result of vanishing order, one can show the
quantitative unique continuation at infinity. The quantitative
unique continuation at infinity is characterized by a lower bound
for $\mathcal{M}\pr{R}$, where
\begin{equation}
\mathcal{M}\pr{R} := \inf_{|x_0|=R}\sup_{\mathbb B_1(x_0)}|u(x)|.
\label{MRDef}
\end{equation}
For the equation (\ref{schro}) in $\mathbb R^n$, it is shown in
\cite{BK05} that
\begin{equation}
\mathcal{M}\pr{R} \ge c \exp\brac{-C R^\frac{4}{3} \log R}
\label{vanish}
\end{equation}
from a scaling argument using the estimates (\ref{like}). We are
able to show the following characterization of solution at infinity
for higher order elliptic equations.
\begin{theorem}
Assume that $\norm{V_\alpha}_{L^\infty\pr{\R^n}} \le A_\alpha$ and
$\norm{V_0}_{L^s\pr{\R^n}} \le A_0$. Let $u$ be a solution to
\eqref{goal} in $\R^n$. Assume that $\norm{u}_{L^\iny\pr{\R^n}} \le
C_0$ and $\abs{u\pr{0}} \ge 1$.

I): In the case of $n> 4m-2$, assume that $s\in (\frac{2n}{3m}, \
\infty]$. Then for $R >> 1$,
\begin{equation*}
\mathcal{M}\pr{R} \ge c \exp\brac{-C R^\Theta \log R},
\end{equation*}
where $\disp \Theta= \left\{\begin{array}{ll}
\frac{2(2m -\alpha_0)}{3m-2\alpha_0} &  \alpha_0 \geq \frac{n}{s} \bigskip \\
\frac{2(2ms-n)}{3ms-2n} & \alpha_0 < \frac{n}{s}
\end{array}\right.$, $c = c\pr{n, m, s, C_0}$, and $C = C\pr{n,
m, s, C_0, A_0, \cdots, A_{\alpha_0}}$.

II): In the case of $ n=4m-2$, assume that $s\in
(\frac{4(2m-1)}{3m}, \, \infty]$, Then for any sufficiently small
constant $\eps>0$ and $R
>> 1$,
\begin{equation*}
\mathcal{M}\pr{R} \ge c \exp\brac{-C R^{\tilde\Theta} \log R},
\end{equation*}
where $\disp \tilde{\Theta} = \left\{\begin{array}{ll}
\frac{2(2m -\alpha_0)}{3m-2\alpha_0} \quad &  \alpha_0 \geq \frac{8m(2m-1)-3mn+4m(2m-1)(s-2)\epsilon}{ms+4(2m-1)-2n+2(2m-1)(s-2)\epsilon} \bigskip \\
\frac{2(2ms-n)}{3ms-4(2m-1)-2(2m-1)(s-2)\epsilon} & \alpha_0 <
\frac{8m(2m-1)-3mn+4m(2m-1)(s-2)\epsilon}{ms+4(2m-1)-2n+2(2m-1)(s-2)\epsilon}
\end{array}\right.$, \\ $c = c\pr{n, s, m, \eps, C_0}$, and $C = C\pr{n, s,
m, \eps, C_0, A_0, \cdots, A_{\alpha_0}}$.

III): In the case of $2\leq n<4m-2$,  assume that $s\in
(\frac{4(2m-1)}{3m}, \, \infty]$. Then
for $R >> 1$,
\begin{equation*}
\mathcal{M}\pr{R} \ge c \exp\brac{-C R^{\bar{\Theta}} \log R},
\end{equation*}
where $\disp \bar{\Theta} = \left\{\begin{array}{ll}
\frac{2(2m -\alpha_0)}{3m-2\alpha_0} \quad &  \alpha_0 \geq \frac{8m(2m-1)-3mn}{ms+4(2m-1)-2n} \bigskip \\
\frac{2(2ms-n)}{3ms-4(2m-1)} & \alpha_0 <
\frac{8m(2m-1)-3mn}{ms+4(2m-1)-2n}
\end{array}\right.$, \\ $c =
c\pr{n, s, m, C_0}$,  and $C = C\pr{n, s, m, C_0, A_0, \cdots,
A_{\alpha_0}}$. \label{UCVW}
\end{theorem}

\begin{remark}
1. In particular, the vanishing at infinity as (\ref{vanish}) for the
higher order elliptic equation (\ref{zhu}) with  $V_0\in L^\infty$ was shown
 in \cite{HWZ16}. If $s=\infty$ and $V_\alpha=0$, our Theorem will implies (\ref{vanish}) as well. Obviously, the results in \cite{HWZ16}
 is just a special case of Theorem \ref{UCVW}. The work enables us to deal with  the singular potential $V_0$ as well as the presence of coefficient function $V_\alpha$ for the results of vanishing at infinity.

2. The case I and II in  Theorem \ref{UCVW} also work for the case $m=1$. If $V_\alpha=0$, $m=1$ and $s=\infty$, the conclusions match the sharp result (\ref{vanish}) for the seconder order elliptic equation in \cite{BK05}. Clearly, we have obtained the results for much more general cases.

\end{remark}

Generally speaking, the frequency function  and Carleman estimates
are two major ways to obtain qualitative and quantitative unique
continuation results for solutions of partial differential
equations.  The frequency function describes the local growth rate
of $u$ and is considered as a local measure of its ``degree" for a
polynomial like function in $\mathbb B_r$. See e.g. \cite{GL86},
\cite{GL87}, \cite{Ku98}, \cite{Zhu16} for the application of
frequency function, to just mention a few. Carleman estimates are
weighted integral inequalities. To obtain the quantitative
uniqueness results for solutions, one usually uses the Carleman
estimates with a special choice of weight functions to obtain some
type of Hadamard's three-ball theorem, then employ ``propagation of
smallness" argument to obtain maximal order of vanishing. In this
paper, we establish a new quantitative $L^2 \to L^p$ Carleman
inequalities with a range of $p$ value for higher order elliptic
operators. We first derive a quantitative $L^2 \to L^2$ Carleman
inequalities involving terms for every order derivative for second
order elliptic operators. Then, using an iterative procedure, we
obtain a quantitative $L^2 \to L^2$ Carleman inequalities for higher
order elliptic operators. The $L^2 \to L^p$ Carleman estimates are
attained from Sobolev embedding and an interpolation argument, which
adapts the idea in \cite{DZ17}.

Let us comment on the organization of the article. Section
\ref{CarlEst} is devoted to obtaining Carleman estimates for the
higher order elliptic operators with singular potential functions
$V_0$ and coefficient functions $V_\alpha$. In Section
\ref{CarlProofs}, the main tool $L^2\to L^p$ Carleman estimates are
established. We also derive some type of quantitative Caccioppoli
inequality and $L^\infty$ type estimates for higher order elliptic
equations. In section \ref{vanOrd}, we deduce three-ball
inequalities from the Carleman estimates and obtain the vanishing
order estimates from the propagation of smallness argument. The
proof of Theorem \ref{UCVW} is presented in Section \ref{QuantUC}.
The letters $c$, $C$, $C_0$ and $C_1$ denote generic positive
constants that do not depends on $u$, and may vary from line to
line. In the paper, the norm $\|V_0\|_{L^s}$ and $\|V_\alpha\|_{L^\infty}$ are assumed to be sufficiently large. Otherwise, we may assume that $\|V_0\|_{L^s}\leq M_0$ and $\|V_\alpha\|_{L^\infty}\leq M_\alpha$ for some sufficiently large $M_0$ and $M_\alpha$. Then we may replace the $\|V_0\|_{L^s}$ by  $M_0$ and the $\|V_\alpha\|_{L^\infty}$ by $M_\alpha$ in Theorem \ref{thh}.

\section{Carleman estimates }
\label{CarlEst}

In this section, we state the crucial tools, the quantitative
 Carleman estimates. Set
$$\phi(r)=\log r+\log(\log r)^2.$$
Let $r=|x-x_0|$. We use the notation $\|u\|_{L^p(r^{-n} dx)}$ to
denote the $L^p$ norm with weight $r^{-n}$, i.e. $$\|u\|_{L^p(r^{-n}
dx)} = \disp \pr{\int |u|^p r^{-n}\, dx}^\frac{1}{p}.$$ Our
quantitative Carleman estimate for the higher order elliptic
operators is stated as follows. Three cases are discussed in term of
the relation of $n$ and $m$.

\begin{theorem}
(I): In the case of $n>4m-2$ and $2 \leq p \leq \frac{2n}{n-4m+2}$,
there exist a constant $C$ and a sufficiently small $R_0$ such that
for any $u\in C^{\infty}_{0}\pr{\mathbb
B_{R_0}(x_0)\backslash\set{x_0} }$ and $\tau>1$, one has
\begin{align}
&\tau^{\be_0} \|e^{-\tau \phi(r)}(\log r)^{-m} u\|_{L^p(r^{-n}dx)} +
\sum^{2m-1}_{|\alpha|=1} \tau^{\be_\alpha} \|e^{-\tau \phi(r)} (\log
r )^{-m} r^{|\alpha|}
 D^\alpha u\|_{L^2(r^{-n}dx)}
\nonumber \medskip\\
&\leq  C \| e^{-\tau \phi(r)} r^{2m} \triangle^m u\|_{L^2(r^{-n}
dx)} , \label{mainth}
\end{align}
where $\be_0 =\frac{3mp-n(p-2)}{2p} $ and $\be_\alpha =
\frac{3m-2|\alpha|}{2}$.

 (II): In the case of $ n= 4m-2$ and $2 \leq p <
 \infty$,
 there exist a constant $C$ and a sufficiently small $R_0$ such that
for any $u\in C^{\infty}_{0}\pr{\mathbb
B_{R_0}(x_0)\backslash\set{x_0} }$ and $\tau>1$, one has
\begin{align}
&\tau^{\tilde{\be}_0}\| e^{-\tau \phi(r)}(\log r)^{-m}u\|_{L^p}  +
\sum^{2m-1}_{|\alpha|=0} \tau^{\be_\alpha} \|e^{-\tau \phi(r)} (\log
r )^{-m} r^{|\alpha|}
 D^\alpha u\|_{L^2(r^{-n}dx)}
\nonumber \medskip\\
&\leq  C \| e^{-\tau \phi(r)} r^{2m} \triangle^m u\|_{L^2(r^{-n}
dx)} , \label{mainth2}
\end{align}
where $\tilde{\be}_0 =\frac{4m-2}{p}(1-\eps)-\frac{m-2}{2} $,
$\be_\alpha = \frac{3m-2|\alpha|}{2}$, and $0<\epsilon<1$ is
sufficiently small.

(III): In the case of $2\leq n<4m-2$ and $2\leq p\leq \infty$
 there exist a constant $C$ and a sufficiently small $R_0$ such that
for any $u\in C^{\infty}_{0}\pr{\mathbb
B_{R_0}(x_0)\backslash\set{x_0} }$ and $\tau>1$, one has
\begin{align}
&\tau^{\bar{\beta}_0}\| e^{-\tau \phi(r)}(\log r)^{-m}u\|_{L^p}  +
\sum^{2m-1}_{|\alpha|=0} \tau^{\be_\alpha} \|e^{-\tau \phi(r)} (\log
r )^{-m} r^{|\alpha|}
 D^\alpha u\|_{L^2(r^{-n}dx)}
\nonumber \medskip\\
&\leq  C \| e^{-\tau \phi(r)} r^{2m} \triangle^m u\|_{L^2(r^{-n}
dx)} , \label{mainth3}
\end{align}
where $\bar{\be}_0 =\frac{4m-2}{p}-\frac{m-2}{2} $,
$\be_\alpha = \frac{3m-2|\alpha|}{2}$.
 \label{CCarlpqp}
\end{theorem}

We provide the proof of Theorem \ref{CCarlpqp} in the next section.
We are going to  use Theorem \ref{CCarlpqp} to establish the
following Carleman estimates for higher order elliptic equations of
the form \eqref{goal}. For an appropriate choice of $p$ and
sufficiently large $\tau$, from Theorem \ref{CCarlpqp}, we replace
the higher order elliptic operator with a higher order elliptic
operator with potential functions $V_0$ and coefficient functions
$V_\alpha$ using H\"older's inequality and the triangle inequality.

\begin{theorem}
I): In the case of $n> 4m-2$, assume that $s\in (\frac{2n}{3m}, \
\infty]$, there exist constants $C_0$, $C_1$, and sufficiently small
$R_0 < 1$ such that for any $u\in C^{\infty}_{0}(\mathbb
B_{R_0}(x_0)\setminus \set{x_0})$ and large positive constant
$$  \tau>C_0(1+\sum_{|\alpha|=1}^{\alpha_0}\|V_\alpha\|_{L^\infty}^\mu+\|V_0\|_{L^s}^\nu),        $$
one has
\begin{align}
&\tau^{\be_0} \|e^{-\tau \phi(r)}(\log r)^{-m} u\|_{L^p(r^{-n}dx)} +
\sum^{2m-1}_{|\alpha|=1} \tau^{\be_\alpha} \|e^{-\tau \phi(r)}(\log
r )^{-m} r^{|\alpha|}
 D^\alpha u\|_{L^2(r^{-n}dx)}
\nonumber \medskip\\
&\leq  C_1 \| e^{-\tau \phi(r)} r^{2m} (\triangle^m u+
\sum_{|\alpha|=1}^{\alpha_0} V_\alpha D^\alpha u+ V_0
u)\|_{L^2(r^{-n} dx)}, \label{main}
\end{align}
where
$$ \mu=\frac{2}{3m-2|\alpha|} \quad \mbox{and} \quad \nu=\frac{2s}{3ms-2n},          $$
 and $\be_0, \beta_\alpha$ and $p$ as defined in Theorem \ref{CCarlpqp}.

II:) In the case $n= 4m-2$, assume that $s\in (\frac{4(2m-1)}{3m},
\, \infty]$, there exist constants $C_0$, $C_1$, and sufficiently
small $R_0 < 1$ such that for any $u\in C^{\infty}_{0}(\mathbb
B_{R_0}(x_0)\setminus \set{x_0})$ and large positive constant
$$  \tau>C_0(1+\sum_{|\alpha|=1}^{\alpha_0}\|V_\alpha\|_{L^\infty}^\mu+\|V_0\|_{L^s}^{\tilde{\nu}}),        $$
one has
\begin{align}
&\tau^{\tilde{\be}_0} \|e^{-\tau \phi(r)}(\log r)^{-m} u\|_{L^p} +
\sum^{2m-1}_{|\alpha|=1} \tau^{\be_\alpha} \|e^{-\tau \phi(r)}(\log
r )^{-m} r^{|\alpha|}
 D^\alpha u\|_{L^2(r^{-n}dx)}
\nonumber \medskip\\
&\leq  C_1 \| e^{-\tau \phi(r)} r^{2m} (\triangle^m u+
\sum_{|\alpha|=1}^{\alpha_0} V_\alpha D^\alpha u+ V_0
u)\|_{L^2(r^{-n} dx)}, \label{main1}
\end{align}
where
$$ \mu=\frac{2}{3m-2|\alpha|} \quad \mbox{and} \quad \tilde{\nu}= \frac{2s}{3ms-4(2m-1)-2(2m-1)(s-2)\epsilon},         $$
 and $\tilde{\be}_0$, $\beta_\alpha$ and $p$ as defined in Theorem
\ref{CCarlpqp}.

III:)  In the case $2\leq n<4m-2$, assume that $s\in (\frac{4(2m-1)}{3m},
\, \infty]$, there exist constants $C_0$, $C_1$, and sufficiently
small $R_0 < 1$ such that for any $u\in C^{\infty}_{0}(\mathbb
B_{R_0}(x_0)\setminus \set{x_0})$ and large positive constant
$$  \tau>C_0(1+\sum_{|\alpha|=1}^{\alpha_0}\|V_\alpha\|_{L^\infty}^\mu+\|V_0\|_{L^s}^{\bar{\nu}}),        $$
one has
\begin{align}
&\tau^{\bar{\be}_0} \|e^{-\tau \phi(r)}(\log r)^{-m} u\|_{L^p} +
\sum^{2m-1}_{|\alpha|=1} \tau^{\be_\alpha} \|e^{-\tau \phi(r)}(\log
r )^{-m} r^{|\alpha|}
 D^\alpha u\|_{L^2(r^{-n}dx)}
\nonumber \medskip\\
&\leq  C_1 \| e^{-\tau \phi(r)} r^{2m} (\triangle^m u+
\sum_{|\alpha|=1}^{\alpha_0} V_\alpha D^\alpha u+ V_0
u)\|_{L^2(r^{-n} dx)}, \label{main3}
\end{align}
where
$$ \mu=\frac{2}{3m-2|\alpha|} \quad \mbox{and} \quad \bar{\nu}= \frac{2s}{3ms-4(2m-1)},         $$
 and $\bar{\be}_0$, $\beta_\alpha$ and $p$ as defined in Theorem
\ref{CCarlpqp}.
\label{CarlpqVW}
\end{theorem}

\begin{proof}
Case I): We first consider the case $n>4m-2$. With the aid of
Theorem \ref{CCarlpqp} and the triangle inequality, we see that
\begin{align}
\indent &\tau^{\be_0} \|e^{-\tau \phi(r)}(\log r)^{-m}
u\|_{L^p(r^{-n}dx)} + \sum^{2m-1}_{|\alpha|=1} \tau^{\be_\alpha}
\|e^{-\tau \phi(r)} (\log r )^{-m} r^{|\alpha|}
 D^\alpha u\|_{L^2(r^{-n}dx)}
\nonumber \medskip\\
&\leq  C \| e^{-\tau \phi(r)} r^{2m} \triangle^m u\|_{L^2(r^{-n}
dx)} \nonumber \\
&\leq  C \| e^{-\tau \phi(r)} r^{2m} (\triangle^m u+
\sum_{|\alpha|=1}^{\alpha_0} V_\alpha D^\alpha u+ V_0
u)\|_{L^2(r^{-n} dx)}\nonumber \\&+\sum_{|\alpha|=1}^{\alpha_0}
C\|e^{-\tau
\phi(r)} r^{2m} V_\alpha D^\alpha u\|_{L^2(r^{-n} dx)} \nonumber \\
& +C\|e^{-\tau \phi(r)} r^{2m} V_0  u\|_{L^2(r^{-n} dx)}.
\label{ineqq}
\end{align}
 Now
we estimate the last two terms in the right hand side of
(\ref{ineqq}). Set $p=\frac{2s}{s-2}$. Note that $m\geq 2$. By the
assumption that  $n>4m-2$, we obtain $s>\frac{2n}{3m}>2$. Thus, $p$
is a positive constant. Since $\frac{n}{2m-1}\leq
\frac{2n}{3m}<s\leq \infty$, then $2\leq p<\frac{2n}{n-3m}\leq
\frac{2n}{n-4m+2}$, which in the range of $p$ in Theorem
\ref{CCarlpqp}.  Following from H\"older's inequality, we obtain
that
\begin{align}
& \| e^{-\tau\phi(r)} r^{2m} V_0 u\|_{L^2(r^{-n} dx)} \nonumber \medskip \\
&\le  \|V_0\|_{L^{s}\pr{B_{R_0}}} \|e^{-\tau\phi(r)} r^{2m +
\frac n p-\frac n 2} u\|_{L^{p}(r^{-n}dx)} \nonumber \\
&\le \|V_0\|_{L^{s}\pr{B_{R_0}}} \|e^{-\tau\phi(r)}\pr{\log
r}^{-m}u\|_{L^p(r^{-n} dx)}\| (\log r)^m r^{2m+
\frac n p-\frac n 2}\|_{L^{\iny}\pr{B_{R_0}}} \nonumber \\
&\le C \|V_0\|_{L^{s}\pr{B_{R_0}}} \| e^{-\tau \phi(r)} (\log
r)^{-m} u\|_{L^p(r^{-n} dx)}, \label{hoddd}
\end{align}
where we have used the fact that $2m + \frac n p-\frac n 2>0$ and
$R_0$ is small.
 Furthermore, using H\"older's
inequality,
\begin{align}
& \sum_{|\alpha|=1}^{\alpha_0} \| e^{-\tau\phi(r)} r^{2m} V_\alpha
D^\alpha u\|_{L^2(r^{-n} dx)} \nonumber \\ &\leq
\sum_{|\alpha|=1}^{\alpha_0} \|V_\alpha\|_{L^\infty}\|
r^{2m-|\alpha|} (\log r)^m\|_{L^\infty} \| e^{-\tau\phi(r)}(\log r)^{-m} r^{|\alpha|} D^\alpha u\|_{L^2(r^{-n}dx)}  \nonumber \\
&\leq \sum_{|\alpha|=1}^{\alpha_0} \|V_\alpha\|_{L^\infty}\|
e^{-\tau\phi(r)}(\log r)^{-m} r^{|\alpha|} D^\alpha
u\|_{L^2(r^{-n}dx)}. \label{hodd}
\end{align}
In order to absorb the last two terms in the right hand side of
(\ref{ineqq}) into the the left hand side, from (\ref{hodd}) and
(\ref{hoddd}), we choose
\begin{equation}
\tau^{\beta_\alpha}\geq c\|V_\alpha\|_{L^\infty} \quad \mbox{and}
\quad \tau^{\beta_0}\geq c\|V_0\|_{L^s}. \label{asso}
\end{equation}
From the assumption of $\alpha_0$, we know
$\beta_\alpha=\frac{3m-2|\alpha|}{2}>0$. We can check that $\be_0
=\frac{3mp-n(p-2)}{2p}=\frac{3ms-2n}{2s}$. Because $s>
\frac{2n}{3m}$, we see that $\beta_0>0$.

 Therefore, to reach (\ref{asso}),  we need to choose
$$  \tau>C(1+\sum_{|\alpha|=1}^{\alpha_0}\|V_\alpha\|_{L^\infty}^\mu+\|V_0\|_{L^p}^\nu)        $$
where
$$ \mu=\frac{2}{3m-2|\alpha|} \quad \mbox{and} \quad \nu=\frac{2s}{3ms-2n}.           $$

Case II): Now we turn to the case of $n= 4m-2$. Carrying out the
similar arguments as (\ref{ineqq}), we obtain
\begin{align}
\indent &\tau^{\tilde{\be}_0} \|e^{-\tau \phi(r)}(\log r)^{-m}
u\|_{L^p} + \sum^{2m-1}_{|\alpha|=0} \tau^{\be_\alpha} \|e^{-\tau
\phi(r)} (\log r )^{-m} r^{|\alpha|}
 D^\alpha u\|_{L^2(r^{-n}dx)}
\nonumber \medskip\\
&\leq  C \| e^{-\tau \phi(r)} r^{2m} (\triangle^m u+
\sum_{|\alpha|=1}^{\alpha_0} V_\alpha D^\alpha u+ V_0
u)\|_{L^2(r^{-n} dx)}\nonumber \\&+\sum_{|\alpha|=1}^{\alpha_0}C
\|e^{-\tau
\phi(r)} r^{2m} V_\alpha D^\alpha u\|_{L^2(r^{-n} dx)} \nonumber \\
& +C\|e^{-\tau \phi(r)} r^{2m} V_0  u\|_{L^2(r^{-n} dx)}.
\label{ineqqq}
\end{align}
Set $p=\frac{2s}{s-2}$ again. Since $s>\frac{4(2m-1)}{3m}$, we check
that $s>2$. It follows from H\"older's inequality that
\begin{align}
& \| e^{-\tau\phi(r)} r^{2m} V_0 u\|_{L^2(r^{-n} dx)} \nonumber \medskip \\
&\le  \|V_0\|_{L^{s}\pr{B_{R_0}}} \|e^{-\tau\phi(r)} r^{2m -\frac n 2} u\|_{L^{p}} \nonumber \\
&\le \|V_0\|_{L^{s}\pr{B_{R_0}}} \|e^{-\tau\phi(r)}\pr{\log
r}^{-m}u\|_{L^p}\| (\log r)^m r^{2m-\frac n 2}\|_{L^{\iny}\pr{B_{R_0}}} \nonumber \\
&\le C \|V_0\|_{L^{s}\pr{B_{R_0}}} \| e^{-\tau \phi(r)} (\log
r)^{-m} u\|_{L^p},\label{lowdim}
\end{align}
where we have considered that $2m-\frac{n}{2}>0$. For the terms
involving higher order derivatives, we carry out the the same
argument as (\ref{hodd}). We can also absorb the last two terms in
the right hand side of (\ref{ineqqq}) into the left hand side.
Together with (\ref{lowdim}) and (\ref{hodd}), we choose
\begin{equation}
\tau^{\beta_\alpha}\geq c\|V_\alpha\|_{L^\infty} \quad \mbox{and}
\quad \tau^{\tilde{\beta}_0}\geq c\|V_0\|_{L^s}.\label{mme}
\end{equation}
Since $s>\frac{4(2m-1)}{3m}$, we can check that
$$\tilde{\beta}_0=\frac{2(2m-1)}{p}(1-\epsilon)-\frac{m-2}{2}=\frac{3ms-4(2m-1)-2(2m-1)(s-2)\eps}{2s}>0$$
by choosing $\epsilon$ sufficiently small.

To satisfy (\ref{mme}), we choose
$$  \tau>C(1+\sum_{|\alpha|=1}^{\alpha_0}\|V_\alpha\|_{L^\infty}^\mu+\|V_0\|_{L^p}^{\tilde{\nu}})        $$
with
$$ \mu=\frac{2}{3m-2|\alpha|} \quad \mbox{and} \quad \tilde{\nu}= \frac{2s}{3ms-4(2m-1)-2(2m-1)(s-2)\epsilon},         $$
where $0<\epsilon<1$ is sufficiently small.

Case III): At last, we deal with the case $2\leq n<4m-2$. Similar to the argument in Case II,
by
triangle inequality and H\"older's inequality, it follows from
Theorem \ref{CCarlpqp} that
\begin{align}
\indent &\tau^{\bar{\be}_0} \|e^{-\tau \phi(r)}(\log r)^{-m}
u\|_{L^p} + \sum^{2m-1}_{|\alpha|=0} \tau^{\be_\alpha} \|e^{-\tau
\phi(r)} (\log r )^{-m} r^{|\alpha|}
 D^\alpha u\|_{L^2(r^{-n}dx)}
\nonumber \medskip\\
&\leq  C \| e^{-\tau \phi(r)} r^{2m} (\triangle^m u+
\sum_{|\alpha|=1}^{\alpha_0} V_\alpha D^\alpha u+ V_0
u)\|_{L^2(r^{-n} dx)}\nonumber \\&+\sum_{|\alpha|=1}^{\alpha_0}C \|V_\alpha \|_{L^\infty}
\|e^{-\tau
\phi(r)} r^{2m} D^\alpha u\|_{L^2(r^{-n} dx)} \nonumber \\
& +C\|e^{-\tau \phi(r)} r^{2m} V_0  u\|_{L^2(r^{-n} dx)}
\label{ineqqqq}
\end{align}
and
\begin{equation}
\| e^{-\tau\phi(r)} r^{2m} V_0 u\|_{L^2(r^{-n} dx)}\leq C \|V_0\|_{L^{s}\pr{B_{R_0}}} \| e^{-\tau \phi(r)} (\log
r)^{-m} u\|_{L^p}\label{lowdimm}
\end{equation}
for $p=\frac{2s}{s-2}.$ We choose
\begin{equation}
\tau^{\beta_\alpha}\geq c\|V_\alpha\|_{L^\infty} \quad \mbox{and}
\quad \tau^{\bar {\beta}_0}\geq c\|V_0\|_{L^s}.\label{mmmme}
\end{equation}
since $s\in (\frac{4(2m-1)}{3m},
\, \infty]$, we can verify that
$$\bar{\beta}_0=\frac{3ms-4(2m-1)}{2s}>0.$$
To satisfy (\ref{mmmme}), we select
$$ \tau>C(1+\sum_{|\alpha|=1}^{\alpha_0}\|V_\alpha\|_{L^\infty}^\mu+\|V_0\|_{L^s}^{\bar\nu}),  $$
where $\mu=\frac{2}{3m-2|\alpha|}$ and $\bar \nu=\frac{2s}{3ms-4(2m-1)}$.
 Thus, the estimate (\ref{main3})
is achieved.

Together with the discussion in those three cases, we complete the
proof of Theorem \ref{CCarlpqp}.
\end{proof}

\section{Proof of Carleman estimates }
\label{CarlProofs}

In this section, we prove the crucial tool in the whole paper, i.e.
the $L^2- L^p$ Carleman estimate stated in Theorem \ref{CCarlpqp}.
To prove our Carleman estimate, we first establish a $L^2$ type
Carleman estimates for higher order elliptic operators.

We introduce polar coordinates in $\mathbb R^n\backslash \{0\}$ by
setting $x=r\omega$, with $r=|x|$ and
$\omega=(\omega_1,\cdots,\omega_n)\in S^{n-1}$. Further, we use a
new coordinate $t=\log r$. Then
$$ \frac{\partial }{\partial x_j}=e^{-t}(\omega_j\partial_t+  \Omega_j), \quad 1\leq j\leq n, $$
where $\Omega_j$ is a vector field in $S^{n-1}$. It is well known
that vector fields $\Omega_j$ satisfy
$$ \sum^{n}_{j=1}\omega_j\Omega_j=0 \quad \mbox{and} \quad
\sum^{n}_{j=1}\Omega_j\omega_j=n-1.$$

 The adjoint of $\Omega_j$ is an operator in $L(S^{n-1})$ given by
 $$ \Omega^\ast_j=(n-1)\omega_j-\Omega_j. $$
 It is known that
 $$\sum^{n}_{j=1}\Omega^\ast_j\Omega_j=-\triangle_\omega.  $$
We denote $\Omega^\alpha$ as the product of
$\Omega_1^{\alpha_1}\cdots \Omega_n^{\alpha_n}$, where
$\alpha=(\alpha_1, \cdots, \alpha_n)\in \mathbb N^n.$

 We are
interested in $\phi(r)\in C^\infty_0(B_{R_0}(x_0) \setminus
\set{x_0})$ for some small $R_0$. Since $r=e^t$, then $r\to 0$ if
and only if $t\to-\infty$. In terms of $t$, we consider the case
$-\infty<t<t_0 < 0$, where $|t_0|$ is chosen to be sufficiently
large. Since
$$ r^2\triangle u=r^2 \partial^2_r u+(n-1)r \partial_r u+\triangle_{\omega} u. $$
 In the new coordinate system, the Laplace operator takes
the form
\begin{equation}
e^{2t} \triangle u=\partial^2_t u+(n-2)\partial_t
u+\triangle_{\omega} u, \label{laplace}
\end{equation}
where $\disp \triangle_\omega=\sum_{j=1}^n \Omega^2_j$ is the
Laplace-Beltrami operator on $S^{n-1}$.

The idea of establishing the following Carleman estimates is
motivated by \cite{R97} and \cite{L07}. However, the test function
$\phi(r)=\frac{\tau}{2}(\log r)^2$ chosen in \cite{R97} and
\cite{L07} is not
 a log linear function.
 The log linearity of test functions is
essential in deriving the vanishing order in section 4. In the
following proposition, we choose  a log linear function
$\phi(r)=\log r+\log (\log r)^2$. More delicate analysis is devoted
to establishing the estimates.

\begin{proposition}
Given $\sigma_1\in \mathbb Z$ and $\sigma_2\in \mathbb Z$.
 Then there
exist positive constant $C$, large positive constant $\tau_0$, and
sufficiently small $R_0 < 1$ such that for any $u\in
C^{\infty}_{0}(\mathbb B_{R_0}(x_0)\setminus \set{x_0})$, one has
\begin{align}
\sum^{2m-1}_{|\alpha|=0} \tau^{\be_\alpha} \|e^{-\tau \phi(r)}(\log
r )^{\sigma_2-m} r^{\sigma_1+|\alpha|}  |D^\alpha
u|\|_{L^2(r^{-n}dx)}
\nonumber \medskip\\
\leq  C \| e^{-\tau \phi(r)} (\log r)^{\sigma_2} r^{\sigma_1+2m}
\triangle^m u\|_{L^2(r^{-n} dx)}  \label{mainCar}
\end{align}
for $\tau>\tau_0$, where $\be_\alpha = \frac{3m-2|\alpha|}{2}$.
 \label{hicarl}
\end{proposition}
\begin{proof}
Our strategy is to prove a $L^2$ type Carleman estimates for second
order elliptic operator. Then we can perform an interative process
to get the Carleman estimates for higher order elliptic operators.
First, we derive the following Carleman estimates involving terms
for every order derivative with weights,
\begin{equation}
\sum_{|\alpha|\leq 2}
\tau^{\frac{3-2|\alpha|}{2}}\|e^{-\tau\phi(r)}(\log r)^{\sigma_2-1}
r^{\sigma_1+|\alpha|}D^\alpha u \|_{L^2(r^{-n}dx)}\leq
\|e^{-\tau\phi(r)}(\log r)^{\sigma_2} r^{\sigma_1+2} \triangle u
\|_{L^2(r^{-n}dx)}. \label{keyy}
\end{equation}

By the polar coordinates, the right hand side of (\ref{keyy}) can be
written as
\begin{align*}
\iint e^{-2\tau \phi(r)} r^{2\sigma_1+4-n}(\log
r)^{2\sigma_2}|\triangle u|^2 dx &=\iint e^{-2\tau (t+\log t^2)}
e^{2\sigma_1t+4t}t^{2\sigma_2}|\triangle u|^2 \, dt d \omega
\\
&=\iint |e^{-\tau (t+\log t^2)}
e^{\sigma_1t+2t}t^{\sigma_2}\triangle u|^2 \, dt d \omega.
\end{align*}
Let $$ u=e^{\tau (t+\log t^2)} e^{-\sigma_1 t} t^{-\sigma_2} v.$$

   Direct calculations show that
\begin{align}
e^{-\tau (t+\ln t^2)} e^{\sigma_1t+2t}t^{\sigma_2}\triangle u &=
[\partial_t+(\tau+2\tau t^{-1})]^2
v+(n-2)\big(\partial_t+(\tau+2\tau t^{-1})\big) v \nonumber
\\&+\triangle_\omega v+ a\partial_t v+bv,
\end{align}
where $a=-2\sigma_1-\frac{2\sigma_2}{t}$ and
$$ b=-(n-2)(\sigma_1+\frac{\sigma_2}{t})+\frac{\sigma_2}{t^2}+(\sigma_1-\frac{\sigma_2}{t})^2
-2(\tau+\frac{2\tau}{t})(\sigma_1+\frac{\sigma_2}{t}).   $$ Note
that \begin{equation} a'(t)=O(t^{-2}),\quad \mbox{and}\quad
b'(t)=O({\tau t^{-2}}).
\end{equation}
 Next we
define a new operator by
$$\triangle_\tau v= [\partial_t+(\tau+2\tau t^{-1})]^2
v+\big((n-2)+a\big)\partial_t v +[(n-2)(\tau+2\tau t^{-1})+b\big]
v+\triangle_\omega v.
$$  To show (\ref{keyy}), it is equivalent to obtain that
\begin{equation}
\sum_{j+|\alpha|\leq 2} \tau^{3-2(j+|\alpha|)}\iint t^{-2}
|\partial_t^j \Omega^\alpha v|^2\,dt d\omega \leq C \iint
|\triangle_\tau v|^2 \,dt d\omega.
\end{equation}

Let $\triangle^-_\tau v $ be the operator obtained from
$\triangle_\tau v$ by replacing $\partial_t$ with $-\partial_t$,
i.e.
$$\triangle_\tau^- v= [\partial_t-(\tau+2\tau t^{-1})]^2
v-\big((n-2)+a\big)\partial_t v+[(n-2)(\tau+2\tau t^{-1})+b\big]
v+\triangle_\omega v.$$ Note that
\begin{align}\triangle_\tau v=
\partial_t^2 v&+ 2(\tau+2\tau t^{-1})\partial_t v- 2t^{-2}\tau
v+(\tau+2\tau
t^{-1})^2v +\big((n-2)+a\big)\partial_t v\nonumber \\
&+[(n-2)(\tau+2\tau t^{-1})+b\big] v+\triangle_\omega v.
\label{starr1}
\end{align}
Similar calculations show that
 \begin{align} \triangle_\tau^- v= \partial_t^2 v&-2(\tau+2\tau t^{-1})\partial_t v+2t^{-2}\tau v+(\tau+2\tau t^{-1})^2v
-\big((n-2)+a\big)\partial_t v \nonumber \\ &+[(n-2)(\tau+2\tau
t^{-1})+b\big]v+\triangle_\omega v. \label{starr2}
\end{align}
On one hand, we compute the integration of the difference of
$\triangle_\tau v$ and $\triangle^-_\tau v$. Define
\begin{equation}
I:=\iint |\triangle_\tau v|^2  \, dtd\omega -\iint |\triangle_\tau^-
v|^2  \, dtd\omega.
\end{equation}
It follows from (\ref{starr1}) and (\ref{starr2}) that
\begin{align}
I=&4\langle\partial_t^2 v+(\tau+2\tau t^{-1})^2 v+[(n-2)(\tau+2\tau
t^{-1})+b\big]v+\triangle_\omega v, \quad 2(\tau+2\tau
t^{-1})\partial_t v \nonumber \\&-2 t^{-2}\tau
v+\big((n-2)+a\big)\partial_t v\rangle, \label{inne}
\end{align}
where $\langle \ ,\rangle$ is denoted as an inner product in $L^2$ space in
$(-\infty, \ \infty)\times S^{n-1}$. We compute each term in the
inner product in the last identity. Integration by parts shows that
\begin{align}
4\langle\partial_t^2 v, \quad 2(\tau+2\tau t^{-1})\partial_tv\rangle &=4\iint
(\tau+2\tau t^{-1})\partial_t(\partial_t v)^2 dtd\omega \nonumber \\
& =8\tau \iint t^{-2} |\partial_t v|^2 \,dtd\omega . \label{start}
\end{align}
Using integration by parts twice, we have
\begin{align}
4\langle\partial_t^2 v, \quad -2t^{-2}\tau v\rangle &=-16\tau \iint
t^{-3}\partial_t v v  \,dtd\omega+8\tau  \iint t^{-2}(\partial_t
v)^2 \, dtd\omega
 \nonumber \\
&= -24 \tau \iint t^{-4} |v|^2  \,dtd\omega+8\tau \iint t^{-2}
|\partial_t v|^2 \,dtd\omega.
\end{align}
Integration by parts indicates that
\begin{align}
4\langle\partial^2_t v, \quad \big((n-2)+a\big)\partial_t v\rangle&=2\iint
\big((n-2)+a\big)
\partial_t|\partial_t v|^2\,dt d\omega \nonumber \\&=-2\iint a'|\partial_t v|^2\,dt d\omega.
\end{align}
Continuing the computation in (\ref{inne}) gives that
\begin{align}
4\langle(\tau+2\tau t^{-1})^2 v, \quad 2(\tau+2\tau t^{-1})\partial_t
v\rangle&=4\iint (\tau+2\tau t^{-1})^3 \partial_t v^2\,dt d\omega \nonumber \\
&=24\tau \iint (\tau+2\tau t^{-1})^2 t^{-2} v^2 \,dt d\omega.
\end{align}
It is clear that
\begin{align}
4\langle(\tau+2\tau t^{-1})^2 v, \quad -2 t^{-2}\tau v\rangle=-8\tau\iint
(\tau+2\tau t^{-1})^2 t^{-2} v^2\,dt d\omega.
\end{align}
From integrations by parts, we get
\begin{align}
4\langle(\tau+2\tau t^{-1})^2 v, \quad \big((n-2)+a\big)
\partial_tv\rangle&=2 \iint
\big((n-2)+a\big)(\tau+2\tau t^{-1})^2\partial_t v^2 \,dt d\omega \nonumber\\
&=8\tau \iint \big((n-2)+a\big)(\tau+2\tau t^{-1})t^{-2} v^2
\,dt d\omega \nonumber\\
&-2\iint a'(\tau+2\tau t^{-1})^2 v^2 \,dt d\omega.
\end{align}
Similar argument yields that
\begin{align}
&4\langle[(n-2)(\tau+2\tau t^{-1})+b] v, \quad 2(\tau+2\tau
t^{-1})\partial_t v\rangle \nonumber \\
&= 4 \iint [(n-2)(\tau+2\tau t^{-1})+b](\tau+2\tau
t^{-1})\partial_t v^2\,dt d\omega \nonumber \\
&=-4\iint [-2(n-2)\tau t^{-2}+b'](\tau+2\tau t^{-1})v^2 \,dt
d\omega\nonumber
\\&+8\tau \iint [(n-2)(\tau+2\tau t^{-1})+b] t^{-2} v^2 \,dt d\omega.
\end{align}
It is obvious that
\begin{align}
&4\langle[(n-2)(\tau+2\tau t^{-1})+b] v, \quad -2t^{-2}\tau v\rangle \nonumber
\\&= -8\tau \iint [(n-2)(\tau+2\tau t^{-1})+b]t^{-2} v^2\,dt d\omega.
\end{align}
It follows from integration by parts that
\begin{align}
&4\langle[(n-2)(\tau+2\tau t^{-1})+b] v, \quad \big((n-2)+a\big)\partial_t
v\rangle \nonumber \\
&=2\iint [(n-2)(\tau+2\tau t^{-1})+b]\big((n-2)+a\big)\partial_t v^2
\,dt d\omega \nonumber \\
&=-2\iint \big(-2(n-2)t^{-2}\tau +b'\big)\big((n-2)+a\big) v^2\,dt
d\omega \nonumber \\&-2 \iint[(n-2)(\tau+2\tau t^{-1})+b]a' v^2\,dt
d\omega.
\end{align}
Since $\triangle_\omega v=\sum_{j=1}^n\Omega_j^2 v$, integration by
parts leads to
\begin{align}
4\langle\triangle_\omega v, \quad 2(\tau+2\tau t^{-1}) \partial_t v\rangle&
=-4\sum_{j=1}^n \iint (\tau+2\tau t^{-1})\partial_t |\Omega_j
v|^2\,dt d\omega
\nonumber \\
&=-8\tau \sum_{j=1}^n \iint t^{-2} |\Omega_j v|^2\,dt d\omega.
\end{align}
It follows that
\begin{align}
4\langle\triangle_\omega v, \quad  -2 \tau t^{-2} v\rangle&=8\tau \sum_{j=1}^n
\iint t^{-2} |\Omega_j v|^2\,dt d\omega.
\end{align}
It is true from integration by parts that
\begin{align}
4\langle\triangle_\omega v, \quad \big((n-2)+a\big)\partial_t v\rangle&= -4
\sum_{j=1}^n \iint \big((n-2)+a\big) \Omega_j v \Omega_j \partial_t
v\,dt d\omega \nonumber \\ &=2\sum_{j=1}^n\iint a' |\Omega_j
v|^2\,dt d\omega. \label{final}
\end{align}
Taking into account of the equalities from (\ref{start}) to
(\ref{final}) gives that
\begin{align}
I\geq \iint {15\tau^3}{t^{-2}} v^2+\tau^2 O({t^{-2}})\,dt d\omega+
15\tau \iint t^{-2} |\partial_t v|^2 \,dt d\omega+2\sum_{j=1}^n\iint
a' |\Omega_j v|^2\,dt d\omega\label{III}
\end{align}
for $-\infty<t<t_0<0$ with $|t_0|$ large enough.

On the other hand, we consider the integration of the sum of
$\triangle_\tau v$ and $\triangle^-_\tau v$. Define
\begin{equation}
J:=\iint t^{-2} |\triangle_\tau v|^2 \, dtd\omega+ \iint t^{-2}
|\triangle_\tau^- v|^2 \, dtd\omega.
\end{equation}
Direction computations from (\ref{starr1}) and (\ref{starr2}) yields
that
\begin{align}
J =&2 \iint t^{-2}\{|\partial_t^2 v|^2+ 4(\tau+2\tau t^{-1})^2
|\partial_t v|^2+ 4t^{-4} \tau^2 v^2+(\tau+2\tau t^{-1})^4
v^2\nonumber\medskip\\& +\big((n-2)+a\big)^2|\partial_t
v|^2+[(n-2)(\tau+2\tau t^{-1})+b]^2 v^2+|\triangle_\omega v|^2\}
\,dtd\omega \nonumber\medskip\\&+ 4\langle t^{-2}
\partial_t^2 v, \quad
(\tau+2\tau t^{-1})^2v+[(n-2)(\tau+2\tau t^{-1})+b]
v+\triangle_\omega v\rangle \nonumber \medskip\\
&+ 4\langle t^{-2}(\tau+2\tau t^{-1})^2 v, \quad [(n-2)(\tau+2\tau
t^{-1})+b]
v+\triangle_\omega v\rangle \nonumber \medskip\\
& +4\langle t^{-2}[(n-2)(\tau+2\tau t^{-1})+b]v, \quad \triangle_\omega v\rangle
\nonumber\medskip\\&+8\langle t^{-2}(\tau+2\tau t^{-1})\partial_t v,  \quad
-2t^{-2}\tau v +\big((n-2)+a\big) v\rangle \nonumber\medskip\\ & +4\langle-2
t^{-4}\tau v, \quad \big((n-2)+a\big)\partial_t v\rangle. \label{anot}
\end{align}
We compute  each inner product in the expression of $J$,
respectively. Integration by parts shows that
\begin{align}
&4<t^{-2}
\partial_t^2 v, \quad
(\tau+2\tau t^{-1})^2v+ [(n-2)(\tau+2\tau t^{-1})+b]v> \nonumber
\\ & =-4\tau^2 \iint t^{-2}|\partial_t v|^2 \, dtd\omega+ \tau^2 \iint
O(t^{-3})|\partial_t v|^2 \, dtd\omega+\tau \iint O(t^{-2})
|\partial_t v|^2\, dtd\omega \nonumber \\
&+12 \tau^2 \iint O(t^{-4}) v^2\, dtd\omega +\tau \iint O(t^{-4})
v^2\, dtd\omega +\tau^2 \iint O(t^{-5}) v^2 \, dtd\omega\nonumber \\
&-4\iint bt^{-2}|\partial_t v|^2\, dtd\omega+2\iint
\partial_t^2(t^{-2}b) v^2\, dtd\omega.
\end{align}

Again, it follows from integration by parts  that
\begin{align}
4\langle t^{-2} \partial_t^2 v, \quad \triangle_\omega v\rangle &= 4\sum_{j=1}^{n}
\iint t^{-2} |\partial_t\Omega_j v|^2\, dtd\omega -8 \iint
t^{-3}\partial_t \Omega_j
v\Omega_j v \, dtd\omega\nonumber \\
&= 4\sum_{j=1}^{n} \iint t^{-2} |\partial_t\Omega_j v|^2\,
dtd\omega-24\iint t^{-4}|\Omega_j u|^2\, dtd\omega.
\end{align}
Direct calculations indicate that
\begin{align}
&4\langle {t^{-2}}(\tau+2\tau t^{-1})^2 v, \quad [(n-2)(\tau+2\tau
t^{-1})+b]v+\triangle_\omega v\rangle \nonumber
\\ &= 4(n-2)\iint t^{-2}(\tau+2\tau
t^{-1})^3 v^2\, dtd\omega \nonumber
\\
&+4\iint t^{-2}(\tau+2\tau t^{-1})b v^2\, dtd\omega \nonumber
\\
&-4 \sum_{j=1}^{n}\iint t^{-2}(\tau+2\tau t^{-1})^2|\Omega_j v|^2\,
dtd\omega.
\end{align}
From the integration by parts, we obtain
\begin{align}
&4\langle {t^{-2}}[(n-2)(\tau+2\tau t^{-1})+b] v, \quad \triangle_\omega v\rangle
\nonumber \\ &=-4\sum_{j=1}^{n} \iint t^{-2}[(n-2)(\tau+2\tau
t^{-1})+b]|\Omega_j v|^2\, dtd\omega.
\end{align}
It follows that
\begin{align}
&8\langle t^{-2}(\tau+2\tau t^{-1})\partial_t v, \quad -2t^{-2}\tau
v+\big((n-2)+a\big)\partial_t v\rangle \nonumber \\ &=\tau^2\iint (-32
t^{-5}-80 t^{-6}) v^2 \,dt d\omega \nonumber \\& +8\iint
t^{-2}(\tau+2\tau t^{-1})\big((n-2)+a\big)|\partial_t v|^2\,
dtd\omega.
\end{align}
Integration by parts yields  that
\begin{align}
&4\langle -2 t^{-4}\tau v, \quad \big((n-2)+a\big)\partial_t v\rangle \nonumber
\\
&=-4\iint t^{-4} \big((n-2)+a\big)\partial_t v^2 \, dtd\omega
\nonumber
\\
&=-16\iint t^{-5}\big((n-2)+a\big)v^2\, dtd\omega+4\iint t^{-4} a'
v^2 \, dtd\omega.
 \label{anothe}
\end{align}

Taking the identities from (\ref{anot}) to (\ref{anothe}) and
$|t_0|$ is sufficiently large into consideration gives that
\begin{align}
J\geq &2 \iint t^{-2}|\partial_t^2 v|^2 \, dtd\omega +\iint \big(
3\tau t^{-2}+\tau^2 O(t^{-3})\big)|\partial_t v|^2 \, dtd\omega
\nonumber
\\ &+
  2\sum^{n}_{j=1} \iint t^{-2}|\triangle_\omega v|^2 \, dtd\omega
 +4  \sum^{n}_{j=1} \iint t^{-2}|\partial_t\Omega_j v| \, dtd\omega
 \nonumber \\&
+\iint \big(\tau^4 t^{-2}+\tau^4 O(t^{-3}) \big)|v|^2 \, dt d\omega
\nonumber \\ &-\sum_{j=1}^n \iint \big( 5\tau^2
t^{-2}+\tau^2O(t^{-3})\big)|\Omega_j v|^2\, dtd\omega. \label{JJJ}
\end{align}

Now we consider the combined effect from $I$ and $J$, i.e. $\tau
I+J$. Using the fact that $-\infty<t<-|t_0|$ with $|t_0|$ large
enough and combining the estimates for $I$ in (\ref{III}) and
estimates for $J$ in (\ref{JJJ}) yields that
\begin{align}
\tau I+J &\geq 15\tau^4 \iint t^{-2}  |v|^2\, dtd\omega+
17\tau^2\iint
 t^{-2}|\partial_t v|^2 \, dtd\omega-6\tau^2\sum^{n}_{j=1} \iint
t^{-2}|\Omega_j
v|^2 \, dtd\omega\nonumber \\
&+2\iint t^{-2} |\partial_t^2 v|^2\, dtd\omega+ \sum^{n}_{j=1}\iint
2 t^{-2} |\triangle_\omega v|^2\, dtd\omega +4\sum^{n}_{j=1}\iint
t^{-2} |\partial_t
\Omega_j v|^2\, dtd\omega \nonumber \\
&=U+ \frac{11}{4}\tau^4\iint  t^{-2}  |v|^2\, dtd\omega+17\tau^2
\iint t^{-2}|\partial_t v|^2\, dtd\omega\nonumber
\\&+\tau^2\sum^{n}_{j=1} \iint t^{-2}|\Omega_j v|^2 \, dtd\omega
+2\iint t^{-2} |\partial_t^2 v|^2\, dtd\omega + \iint t^{-2}
|\triangle_\omega v|^2\, dtd\omega\nonumber
\\& +4\sum^{n}_{j=1}\iint t^{-2}
|\partial_t \Omega_j v|^2\, dtd\omega, \label{combin}
\end{align}
where
$$U= \frac{49}{4}\iint   t^{-2} \tau^4 |v|^2\, dtd\omega -7\tau^2\sum^{n}_{j=1} \iint t^{-2}|\Omega_j
v|^2\, dtd\omega + \sum^{n}_{j=1}\iint t^{-2} |\triangle_\omega
v|^2\, dtd\omega. $$ Note from integration by parts that
$$-\tau^2\sum^{n}_{j=1} \iint t^{-2}|\Omega_j
v|^2\, dtd\omega = \tau^2 \iint t^{-2}\triangle_\omega v v\,
dtd\omega.
$$ Thus, $U\geq 0$ by Cauchy-Schwartz inequality.

By the ellipticity of $\triangle_\omega $, there exists a constant
$C$ such that
\begin{equation}
\sum_{|\alpha|=2} \iint t^{-2}|\Omega^\alpha v|^2
\,dtd\omega\leq C\iint t^{-2}|\triangle_\omega v|^2 \,dtd\omega.
\end{equation}
It follows from (\ref{combin})  that
\begin{align}
C(\tau I+J )&\geq \tau^4\iint  t^{-2}  |v|^2\, dtd\omega+
\tau^2\iint
 t^{-2}|\partial_t v|^2\, dtd\omega +\tau^2\sum^{n}_{j=1} \iint t^{-2}|\Omega_j v|^2 \, dtd\omega \nonumber
\\&+\iint t^{-2}
|\partial_t^2 v|^2\, dtd\omega + \sum_{|\alpha|=2} \iint
t^{-2}|\Omega^\alpha v|^2 \,dtd\omega \nonumber
\\&+\sum^{n}_{j=1}\iint t^{-2}
|\partial_t \Omega_j v|^2\, dtd\omega.
\end{align}
Since
\begin{align}
(\tau I+J )& =\tau \|\triangle_\tau v\|^2_{L^2}-\tau
\|\triangle_\tau^- v\|^2_{L^2}+ \|t^{-1}\triangle_\tau
v\|_{L^2}^2+\|t^{-1}\triangle_\tau^- v\|_{L^2}^2 \nonumber \medskip \\
&\leq (\tau+1) \|\triangle_\tau v\|_{L^2}^2,
\end{align}
then
\begin{align}
(\tau+1)\|\triangle_\tau v\|^2& \geq \tau^4\iint  t^{-2}  |v|^2\,
dtd\omega+ \tau^2\iint  t^{-2}|\partial_t v|^2\, dtd\omega+
\tau^2\sum^{n}_{j=1} \iint t^{-2}|\Omega_j v|^2 \, dtd\omega
\nonumber
\\&+\iint
t^{-2} |\partial_t^2 v|^2\, dtd\omega+\sum_{|\alpha|=2} \iint
t^{-2}|\Omega^\alpha v|^2 \,dtd\omega  \nonumber
\\&+\sum^{n}_{j=1}\iint t^{-2}
|\partial_t \Omega_j v|^2\, dtd\omega.
\end{align}
That is,
\begin{equation}
\|\triangle_\tau v\|_{L^2}^2\geq C\sum_{j+|\alpha|\leq 2}
\tau^{3-2(j+|\alpha|)}\|t^{-1} \partial_t^j\Omega^\alpha
v\|_{L^2}^2. \label{reach}
\end{equation}
Returning to the coordinate $(r, \omega)$ and $u$, the latter
Carleman estimate (\ref{reach}) is equivalent to the following
\begin{align}
&C\sum_{|\alpha|\leq 2}
\tau^{\frac{3-2|\alpha|}{2}}\|e^{-\tau\phi(r)}(\log r)^{\sigma_2-1}
r^{\sigma_1+|\alpha|}|D^\alpha u| \|_{L^2(r^{-n}dx)}\nonumber \\
&\leq \|e^{-\tau\phi(r)}(\log r)^{\sigma_2} r^{\sigma_1+2} \triangle
u \|_{L^2(r^{-n}dx)}. \label{itera}
\end{align}
To get the Carleman estimates for higher order elliptic operators,
we iterate the estimate (\ref{itera}). Let us consider $m=2$ for example.
From (\ref{itera}), we obtain
\begin{align}
&\|e^{-\tau\phi(r)}(\log r)^{\sigma_2} r^{\sigma_1+4} \triangle^2
u \|_{L^2(r^{-n}dx)}  \nonumber \\
&\geq  C\sum_{|\alpha_1|\leq 2}
\tau^{\frac{3-2|\alpha_1|}{2}}\|e^{-\tau\phi(r)}(\log r)^{\sigma_2-1}
r^{\sigma_1+2+|\alpha_1|}|\triangle D^{\alpha_1} u| \|_{L^2(r^{-n}dx)} \nonumber \\
&\geq  C  \sum_{|\alpha_1|\leq 2} \sum_{|\alpha_2|\leq 2} \tau^{\frac{3-2|\alpha_1|}{2}}\tau^{\frac{3-2|\alpha_2|}{2}}
\|e^{-\tau\phi(r)}(\log r)^{\sigma_2-2}
r^{\sigma_1+|\alpha_1|+|\alpha_2|}| D^{\alpha_1+\alpha_2} u| \|_{L^2(r^{-n}dx)} \nonumber \\
&=C\sum_{|\alpha|\leq 4}
\tau^{\frac{6-2|\alpha|}{2}}\|e^{-\tau\phi(r)}(\log r)^{\sigma_2-2}
r^{\sigma_1+|\alpha|}|D^\alpha u| \|_{L^2(r^{-n}dx)}
\end{align}
for $\alpha=\alpha_1+\alpha_2$. Thus, the Carleman estimates (\ref{mainCar}) in the case of $m=2$ is shown.
Note that $\sigma_1$ and $\sigma_2$ are defined to be integers. In
each iteration, $\sigma_1$ and $\sigma_2$ represent different
integers, which is also the reason the iteration is able to be
carried out.
Iterating the estimate (\ref{itera})
$m$ times, it implies that
\begin{align}
&C\sum_{|\alpha|\leq 2m}
\tau^{\frac{3m-2|\alpha|}{2}}\|e^{-\tau\phi(r)}(\log r)^{\sigma_2-m}
r^{\sigma_1+|\alpha|}|D^\alpha u| \|_{L^2(r^{-n}dx)}\nonumber
\\& \leq \|e^{-\tau\phi(r)}(\log r)^{\sigma_2} r^{\sigma_1+2m}
\triangle^m u \|_{L^2(r^{-n}dx)}.\label{explain}
\end{align}
Therefore, we complete the proof of Proposition \ref{hicarl}.
\end{proof}

Based on the quantitative $L^2$ type Carleman estimates in
Proposition \ref{hicarl}, we are going to establish a $L^2\to L^p$
type Carleman estimates to deal with singular weight potentials. The
idea is to use Sobolev embedding and an interpolation argument
inspired by the idea in \cite{DZ17}.

\begin{proof}[Proof of Theorem \ref{CCarlpqp}]
 In particular, when $\sigma_1=\sigma_2=0$, the Carleman estimates
 (\ref{mainCar}) in Proposition \ref{hicarl} takes the form
\begin{equation}
\sum_{|\alpha|=0}^{2m-1}
\tau^{\frac{3m-2|\alpha|}{2}}\|e^{-\tau\phi(r)}(\log r)^{-m}
r^{|\alpha|}|D^\alpha u| \|_{L^2(r^{-n}dx)}\leq
\|e^{-\tau\phi(r)}r^{2m}  \triangle^m u \|_{L^2(r^{-n}dx)}.
\label{carle}
\end{equation}

Case I): Now we focus on the norm that only involves $u$.  For the
case of $n>4m-2$, since $u\in C^\infty_0(\mathbb B_R(x_0)\setminus
\{x_0\})$, by Sobolev embedding $W^{2m-1, 2}\hookrightarrow
L^{\frac{2n}{n-2(2m-1)}}$, we obtain the following
\begin{align}
&\|e^{-\tau\phi(r)}(\log r)^{-m} u \|_{L^\frac{2n}{n-2(2m-1)}(r^{-n}
dx)} \nonumber
\\ & =\|e^{-\tau\phi(r)}(\log r)^{-m}  r^{-\frac{n-2(2m-1)}{2}}
u\|_{L^\frac{2n}{n-2(2m-1)}} \nonumber \\ & \leq C \|
D^{2m-1}[e^{-\tau\phi(r)}(\log r)^{-m}
r^{-\frac{n-2(2m-1)}{2}}u]\|_{L^2} \nonumber \\
&=C\sum_{|\alpha_1|+|\alpha_2|+|\alpha_3|+|\alpha_4|=2m-1}
\tau^{|\alpha_4|} \|e^{-\tau\phi(r)} r^{-|\alpha_4|} (\log
r)^{-m-|\alpha_3|}r^{-|\alpha_3|}\nonumber\\
&\times  r^{-\frac{n-2(2m-1)}{2}-|\alpha_2|} |D^{\alpha_1}
u|\|_{L^2}\nonumber
\\
&\leq \sum_{|\alpha_1|+|\alpha_4|\leq 2m-1} \tau^{|\alpha_4|}
\|e^{-\tau\phi(r)}
 r^{-\frac{n}{2}+|\alpha_1|}(\log
r)^{-m}|D^{\alpha_1} u|\|_{L^2}\nonumber
\\
&= \sum_{|\alpha_1|+|\alpha_4|\leq 2m-1} \tau^{|\alpha_4|}
\tau^{\frac{-3m+2|\alpha_1|}{2}}
 \|e^{-\tau\phi(r)}(\log r)^{-m} r^{2m}\lap^m u
\|_{L^2(r^{-n}dx)} \nonumber \\
&\leq C  \tau^{\frac{m}{2}-1}
 \|e^{-\tau\phi(r)} r^{2m}\lap^m u
\|_{L^2(r^{-n}dx)}, \label{rll}
\end{align}
where we have used that $\phi'(r)=\frac{1}{r}+\frac{2}{r\log r} \le
\frac 1 r$ since $r \le R_0 \le 1$ and (\ref{carle}). It is known
from (\ref{mainCar}) that
\begin{equation}
\|e^{-\tau\phi(r)}(\log r)^{-m} u \|_{L^2(r^{-n}dx)}\leq C
\tau^{\frac{-3m}{2}}  \|e^{-\tau\phi(r)} r^{2m}\lap^m u|
\|_{L^2(r^{-n}dx)}. \label{ll2}
\end{equation}

We are going to do an interpolation argument with the last two
inequalities. Choose $\lambda \in \pr{0,1}$ so that $p = 2 \lambda +
\pr{1-\lambda} \frac{2n}{n-4m+2}$. By H\"older's inequality,
\begin{align*}
\|e^{-\tau \phi(r)}(\log r)^{-m} u\|_{L^p(r^{-n}dx)} &\le \|e^{-\tau
\phi(r)}(\log r)^{-m} u\|_{L^2(r^{-n}dx)}^{\frac{2\la}{p}}
\|e^{-\tau \phi(r)}(\log r)^{-m}
u\|_{L^{\frac{2n}{n-4m+2}}(r^{-n}dx)}^{\frac{2n\pr{1-\la}}{{(n-4m+2)p}}}.
\end{align*}
Since $\la  = \frac{2n-p(n-4m+2)}{8m-4}$, if we set $\te =
\frac{2\la}{p}$, then $$\theta =
\frac{2}{p}\cdot\frac{2n-p(n-4m+2)}{8m-4}=\frac{2n-p(n-4m+2)}{(4m-2)p},$$
and $1 - \te = \frac{(p-2)n}{\pr{4m-2}p}$.  Therefore, from
(\ref{rll}) and (\ref{ll2}),
\begin{align*}
& \|e^{-\tau \phi(r)}(\log r)^{-m} u\|_{L^p(r^{-n}dx)} \medskip\\
&\leq  \|e^{-\tau \phi(r)}(\log r)^{-m}
u\|_{L^{2}(r^{-n}dx)}^{\theta} \|e^{-\tau\phi(r)}(\log
r)^{-m} u\|_{L^\frac{2n}{n-4m+2}(r^{-n}dx)}^{1-\theta} \medskip\\
&\le \brac{C\tau^{\frac{-3m}{2}} \| e^{-\tau \phi(r)}
r^{2m}\triangle^m u\|_{L^2(r^{-n} dx)}}^\te
\brac{C \tau^{\frac{m-2}{2}}  \| e^{-\tau \phi(r)} r^{2m} \triangle^m u\|_{L^2(r^{-n} dx)}}^{1 - \te} \medskip\\
&= C \tau^{-\frac{3m}{2}\theta+\frac{m-2}{2}(1-\theta)} \| e^{-\tau
\phi(r)} r^{2m} \triangle^m u\|_{L^2(r^{-n} dx)} \medskip\\
&= C \tau^{\frac{m-2}{2}-\frac{2n-p(n-4m+2)}{2p}} \| e^{-\tau
\phi(r)} r^{2m} \triangle^m u\|_{L^2(r^{-n} dx)}.
\end{align*}
That is, for any $2\leq p\leq \frac{2n}{n-4m+2}$, we have
\begin{equation}
\tau^{\frac{3m}{2}-\frac{n(p-2)}{2p}}\|e^{-\tau \phi(r)}(\log
r)^{-m} u\|_{L^p(r^{-n}dx)}\leq C \| e^{-\tau \phi(r)} r^{2m}
\triangle^m u\|_{L^2(r^{-n} dx)}.
\end{equation}

It is obvious from \eqref{carle} that
\begin{equation}
\sum_{|\alpha|=1}^{ 2m-1}
\tau^{\frac{3m-2\alpha}{2}}\|e^{-\tau\phi(r)}(\log r)^{-m}
r^{|\alpha|}D^\alpha u \|_{L^2(r^{-n}dx)}\leq C
\|e^{-\tau\phi(r)}r^{2m} \triangle^m u| \|_{L^2(r^{-n}dx)}.
\end{equation}

The combination of the previous two inequalities yields that
\begin{align}
\tau^{\frac{3m}{2}-\frac{n(p-2)}{2p}}\|(\log r)^{-m} e^{-\tau
\phi(r)}u\|_{L^p(r^{-n}dx)} &+\sum_{|\alpha|=1}^{ 2m-1}
\tau^{\frac{3m-2\alpha}{2}}\|e^{-\tau\phi(r)}(\log r)^{-m}
r^{|\alpha|}|D^\alpha u| \|_{L^2(r^{-n}dx)} \nonumber \\ &\leq C
\|e^{-\tau\phi(r)}r^{2m} | \triangle^m u| \|_{L^2(r^{-n}dx)}.
\end{align}
This completes the proof of (\ref{mainth}).

 Case II): For the case $n=4m-2$, we use the same idea as before. By Sobolev embedding
$W^{2m-1, 2}\hookrightarrow {L^{p'}}$ for $2\leq p'<\infty$, we
obtain that
\begin{align}
&\|e^{-\tau\phi(r)}(\log r)^{-m} u \|_{L^{p'}} \nonumber \\ & \leq C
\| D^{2m-1}[e^{-\tau\phi(r)}(\log r)^{-m}
u]\|_{L^2} \nonumber \\
&=C\sum_{|\alpha_1|+|\alpha_2|+|\alpha_3|+|\alpha_4|=2m-1}
\tau^{|\alpha_1|} \|e^{-\tau\phi(r)} r^{-|\alpha_1|}(\log
r)^{-m-|\alpha_2|}r^{-|\alpha_2|-|\alpha_4|} D^{\alpha_3}
u\|_{L^2}\nonumber
\\
&\leq C\sum_{|\alpha_1|+|\alpha_2|+|\alpha_3|\leq 2m-1}
\tau^{|\alpha_1|} \|e^{-\tau\phi(r)} r^{-(2m-1)+\frac{n}{2}}(\log
r)^{-m-|\alpha_2|}r^{|\alpha_3|} D^{\alpha_3} u\|_{L^2(r^{-n}dx)}\nonumber \\
&\leq C\sum_{|\alpha_1|+|\alpha_3|\leq 2m-1} \tau^{|\alpha_1|}
\|e^{-\tau\phi(r)} (\log r)^{-m}r^{|\alpha_3|} D^{\alpha_3}
u\|_{L^2(r^{-n}dx)}\nonumber
\\
&\leq C \sum_{|\alpha_1|+|\alpha_3|\leq 2m-1} \tau^{|\alpha_1|}
\tau^{\frac{-3m+2|\alpha_3|}{2}}
 \|e^{-\tau\phi(r)} r^{2m}\lap^m u
\|_{L^2(r^{-n}dx)} \nonumber \\
&\leq C  \tau^{\frac{m}{2}-1}
 \|e^{-\tau\phi(r)} r^{2m}\lap^m u
\|_{L^2(r^{-n}dx)}, \label{cccc}
\end{align}
where we have used the fact $n=4m-2$ and (\ref{carle}). From
(\ref{mainCar}), we have
\begin{equation}
\|e^{-\tau\phi(r)}(\log r)^{-m} u \|_{L^2}\leq C
\tau^{\frac{-3m}{2}}  \|e^{-\tau\phi(r)} r^{2m}\lap^m u|
\|_{L^2(r^{-n}dx)}. \label{llll}
\end{equation}

As before, we interpolate the last two inequalities. Choose $\la \in
\pr{0,1}$ so that $p = 2 \la + \pr{1-\la} p'$. Note that $2<p<p'$.
The H\"older's inequality implies that
\begin{align*}
\|e^{-\tau \phi(r)}(\log r)^{-m} u\|_{L^p} &\le \|e^{-\tau
\phi(r)}(\log r)^{-m} u\|_{L^2}^{\frac{2\la}{p}} \|e^{-\tau
\phi(r)}(\log r)^{-m} u\|_{L^{p'}}^{\frac{p'\pr{1-\la}}{p}}.
\end{align*}
Since $\la  = \frac{p'-p }{p'-2}$, if we set $\te = \frac{2\la}{p} =
\frac{2(p'-p)}{p(p'-2)}$, then $$1 - \te =
1-\frac{2(p'-p)}{p(p'-2)}= \frac{p'(p-2)}{(p'-2)p}$$ and we have
that $0\leq \theta\leq 1$. Therefore,
\begin{align*}
& \| e^{-\tau \phi(r)}(\log r)^{-m}u\|_{L^p} \\
&\leq  \| e^{-\tau \phi(r)}(\log r)^{-m}u\|_{L^{{2}}}^{\theta} \|
e^{-\tau\phi(r)}(\log
r)^{-m}u\|_{L^{p'}}^{1-\theta} \\
&\le \brac{C\tau^{\frac{-3m}{2}} \|e^{-\tau \phi(r)} r^{2m}
\triangle^m u\|_{L^2(r^{-n} dx)}}^\te
\brac{C \tau^{\frac{m-2}{2}}  \| e^{-\tau \phi(r)} r^{2m} \triangle^m u\|_{L^2(r^{-n} dx)}}^{1 - \te} \\
&= C \tau^{\frac{m-2}{2}-\frac{(4m-2)\theta}{2}} \|e^{-\tau \phi(r)}
r^{2m} \triangle^m
u\|_{L^2(r^{-n} dx)} \\
&=C \tau^{\frac{m-2}{2}-\frac{(4m-2)}{2}\frac{2(p'-p)}{p(p'-2)}} \|
e^{-\tau \phi(r)} r^{2m} \triangle^m u\|_{L^2(r^{-n} dx)}  \\
&=C \tau^{\frac{m-2}{2}-\frac{(4m-2)}{p}(1-\eps)} \| e^{-\tau
\phi(r)} r^{2m} \triangle^m u\|_{L^2(r^{-n} dx)},
\end{align*}
where we have used \eqref{cccc} and \eqref{llll}. Here
$\epsilon=\frac{p-2}{p'-2}$. Since $ p<p'<\infty$, then
$0<\epsilon<1$. Moreover, $p'$ can be any sufficiently large
constant that approaches to infinity so that $\epsilon$ can be
chosen to be any sufficiently small constant that approaches $0$.
Thus, for any $2<p<\infty$,
\begin{equation}
\tau^{\frac{4m-2}{p}(1-\eps)-\frac{m-2}{2}}\|(\log r)^{-m} e^{-\tau
\phi(r)}u\|_{L^p} \leq C \|e^{-\tau \phi(r)} r^{2m} \triangle^m
u\|_{L^2(r^{-n} dx)}. \label{interr}
\end{equation}
Combing (\ref{carle}) with the last inequality gives the proof of
(\ref{mainth2}) in Theorem \ref{CCarlpqp}.

Case III): For the case of $2\leq n< 4m-2$, by the Sobolev embedding inequality, $W^{2m-1, 2}\hookrightarrow {L^\infty}$,
 we have
\begin{align}
\|e^{-\tau\phi(r)}(\log r)^{-m} u \|_{L^{\infty}} \leq C
\| D^{2m-1}[e^{-\tau\phi(r)}(\log r)^{-m}
u]\|_{L^2}.  \nonumber
\end{align}
The similar arguments as Case I and II shows that
\begin{align}
\|e^{-\tau\phi(r)}(\log r)^{-m} u \|_{L^{\infty}} \leq C \tau^{\frac{m}{2}-1}
\| e^{-\tau\phi(r)}r^{2m}
\triangle^m u\|_{L^2(r^{-n} dx)}.
\label{game}
\end{align}
The estimates (\ref{mainth3}) gives that
\begin{equation}
\|e^{-\tau\phi(r)}(\log r)^{-m} u \|_{L^2}\leq C
\tau^{\frac{-3m}{2}}  \|e^{-\tau\phi(r)} r^{2m}\lap^m u|
\|_{L^2(r^{-n}dx)}. \label{game2}
\end{equation}
Using H\"older's inequality, for any $2\leq p\leq \infty$, we interpolate the inequality (\ref{game}) and (\ref{game2}),
\begin{align*}
& \| e^{-\tau \phi(r)}(\log r)^{-m}u\|_{L^p} \\
&\leq  \| e^{-\tau \phi(r)}(\log r)^{-m}u\|_{L^{{2}}}^{\frac{2}{p}} \|
e^{-\tau\phi(r)}(\log
r)^{-m}u\|_{L^{\infty}}^{1-\frac{2}{p}} \\
&\le \brac{C\tau^{\frac{-3m}{2}} \|e^{-\tau \phi(r)} r^{2m}
\triangle^m u\|_{L^2(r^{-n} dx)}}^{\frac{2}{p}}
\brac{C \tau^{\frac{m-2}{2}}  \| e^{-\tau \phi(r)} r^{2m} \triangle^m u\|_{L^2(r^{-n} dx)}}^{1-\frac{2}{p}} \\
&\leq \tau^{\frac{m-2}{2}-\frac{(4m-2)}{p}} \| e^{-\tau
\phi(r)} r^{2m} \triangle^m u\|_{L^2(r^{-n} dx)}.
\end{align*}
Together (\ref{carle}) with the last inequality gives the proof of
(\ref{mainth3}) in Theorem \ref{CCarlpqp}.

Finally, we arrive at the proof of the three cases in Theorem
\ref{CCarlpqp}.
\end{proof}

\section{Vanishing order}
\label{vanOrd}

In this section, we show the vanishing order of the solutions. For
the preparations, we need some kind of Caccioppoli inequality and a
$L^\infty$ bound estimate. We first prove a quantitative type
Caccioppoli inequality for the higher order elliptic equation. In
such inequality, we want to know how the coefficient depends on the
norms of the potential $V_0$ and coefficient functions $V_\alpha$.
The idea is adapted from the Corollary 17.1.4 in the classical book
of H\"ormander \cite{H85}. More delicate analysis is required to
take care of appearance of singular potential $V_0$.

\begin{lemma} Let $u$ be the solution in (\ref{goal}). There exists
a positive constant $C$ that does not depend of $u$ such that
\begin{equation}
\sum_{|\alpha|=0}^{ 2m-1}\|r^{|\alpha|} D^\alpha u\|_{L^2(\mathbb
B_{c_2R_0}\setminus \mathbb B_{c_3R_0})}\leq C (\sum_{|\alpha|=1}^{
\alpha_0}\|V_{\alpha}\|_{L^\infty}+\|V_0\|_{L^s}+1  )^{2m-1} \|
u\|_{L^2(\mathbb B_{c_1R_0}\setminus \mathbb B_{c_4R_0})}
\end{equation}
for all positive constants $0<c_4<c_3<c_2<c_1$. \label{hormmm}
\end{lemma}
\begin{proof}
From Corollary 17.3 in \cite{H85}, it holds that
\begin{align}
\| d^{|\alpha|}(x) D^\alpha u\|_{L^2(X)}\leq C(\| d^{2m}(x)
\triangle^m u\|_{L^2(X)}+\|u\|_{L^2(X)})^{\frac{|\alpha|}{2m}}
\|u\|_{L^2(X)}^{1-\frac{|\alpha|}{2m}}, \label{horm}
\end{align}
where $d(x)$ is the distance from $X$ to $X^c$ and $X^c$ is the
complement of $X$. We apply the estimate (\ref{horm}) with
$X=\mathbb B_{c_1R_0}\setminus \mathbb B_{c_4R_0}$ for some small
$R_0$ with $0<c_4<c_1$. Since $u$ is the solution of the equation
(\ref{goal}),
\begin{align}
d^{2m}\triangle^m u&= d^{2m}( \sum_{|\alpha|=1}^{\alpha_0}
V_{\alpha}
D^\alpha u+ V_0 u) \nonumber \\
&= \sum_{|\alpha|=1}^{\alpha_0 } d^{2m-|\alpha|}(x) d^{|\alpha|}(x)
V_{\alpha} D^\alpha u+d^{2m}V_0u. \label{ho11}
\end{align}
From H\"older's inequality, it follows that
\begin{equation}
\|d^{2m}(x) V_0 u\|_{L^2(X)}\leq \|V_0\|_{L^s(X)}\|d^{2m}(x)
u\|_{L^{\frac{2s}{s-2}}(X)}. \label{cca}
\end{equation}
We first consider the case $n>4m-2$.
 Since $\frac{n}{2m-1}\leq \frac{2n}{3m} < s\leq \infty$, then $2\leq
\frac{2s}{s-2}<\frac{2n}{2n-3m}\leq \frac{2n}{n-4m+2}$. Note that
$d(x)$ is a smooth function if $R_0$ is small. From the Sobolev
embedding $W^{2m-1, 2}\hookrightarrow L^{\frac{2n}{n-2(2m-1)}}$, we
obtain
\begin{equation}
\| d^{2m}u\|_{L^{\frac{2s}{s-2}}(X)}\leq C(
\sum_{|\alpha|=2m-1}\|D^{\alpha} (d^{2m} u)\|_{L^2(X)}+\|
d^{2m}u\|_{L^2(X)}). \label{imbed}
\end{equation}
In the case of $n= 4m-2$, since $s>\frac{4(2m-1)}{3m}$, then $s>2$.
From the Sobolev embedding  $W^{2m-1, 2}\hookrightarrow L^{p}$, for
any $2\leq p<\infty$, we can also obtain (\ref{imbed}) with a
different constant $C$ that does not depend on $u$. It follows from
(\ref{cca}) and (\ref{imbed}) that
\begin{align}
\|d^{2m}(x) V_0 u\|_{L^2(X)}&\leq C
\|V_0\|_{L^s(X)}(\sum_{|\alpha|=0}^{ 2m-1}\| d^{1+|\alpha|}D^\alpha
u\|_{L^2(X)}+\| d^{2m}u\|_{L^2(X)}
)\nonumber \\
&\leq C\|V_0\|_{L^s(X)}(\sum_{|\alpha|=0}^{ 2m-1}\|
d^{1+|\alpha|}D^\alpha u\|_{L^2(X)}). \label{used}
\end{align}
In the case $2\leq n<4m-2$, similar arguments as before show that
\begin{equation}
\|d^{2m}(x) V_0 u\|_{L^2(X)}\leq C\|V_0\|_{L^s(X)}(\sum_{|\alpha|=0}^{ 2m-1}\|
d^{1+|\alpha|}D^\alpha u\|_{L^2(X)}).\label{usedd}
\end{equation}
 Choose $$\beta=\sum_{ |\alpha|=1}^{
2m-1}\|V_{\alpha}\|_{L^\infty}+\|V_0\|_{L^s}+1.$$ Using the
estimates (\ref{horm}) and (\ref{ho11}), it follows that
\begin{align} \| d^{|\alpha|}(x) D^{\alpha} u\|_{L^2(X)}&\leq C\big(
\sum_{ |\alpha|=1}^{2m-1}
R_0^{2m-\alpha}\|V_\alpha\|_{L^\infty}\|d^{|\alpha|}(x) D^\alpha
u\|_{L^2(X)}\nonumber \\ &+\|d^{2m}V_0
u\|_{L^2(X)}+\|u\|_{L^2(X)}\big)^{\frac{|\alpha|}{2m}}
\|u\|_{L^2(X)}^{1-\frac{|\alpha|}{2m}} \nonumber \\
&\leq C(\beta \sum_{|\alpha|=0}^{ 2m-1}\|d^{|\alpha|}(x) D^\alpha
u\|_{L^2(X)})^{\frac{|\alpha|}{2m}}\|u\|_{L^2(X)}^{1-\frac{|\alpha|}{2m}},
\label{cacc}
\end{align}
where we applied the estimate (\ref{used}) for the case $n\geq 4m-2$
and (\ref{usedd}) for the case $2\leq n<4m-2$ in last inequality,
and the fact that $R_0$ is small. Let
$$ S=\sum_{|\alpha|=0}^{2m-1}\|d^{|\alpha|}(x) D^\alpha
u\|_{L^2(X)}.$$ Taking the sum for $0\leq \alpha\leq 2m-1$, from
(\ref{cacc}), we obtain
\begin{align}
S &\leq C\sum_{|\alpha|\leq 2m-1}( \beta
S)^{\frac{|\alpha|}{2m}}\|u\|_{L^2(X)}^{1-\frac{|\alpha|}{2m}} \nonumber \\
&\leq C( \beta S)^{\frac{2m-1}{2m}}\|u\|_{L^2(X)}^{\frac{1}{2m}}.
\end{align}
Thus,
\begin{equation}
S\leq C\beta^{2m-1}\|u\|_{L^2(X)}. \label{cacc1}
\end{equation}
From (\ref{cacc1}), we have
\begin{equation}
\sum_{|\alpha|=0}^{ 2m-1} \int_{ c_2 R_0\leq |d(x)|\leq c_3 R_0} |
d^{|\alpha|}(x) D^\alpha u|^2\,dx\leq C\beta^{4m-2} \int_{ c_4
R_0\leq |d(x) |\leq c_1 R_0} | u|^2\,dx
\end{equation}
for all $0<c_4<c_3<c_2<c_1$. This completes the proof of the lemma.
\end{proof}

We need to establish a $L^\infty$-version of  three-ball theorem.
However, it seems that the classical De Giorgi-Nash-Moser theory
does not exist for higher order elliptic equations. We will deduce
the estimate by Sobolev embedding and a $W^{2m,p}$ type estimate. We
first present a $W^{2m,p}$ type estimates for higher order elliptic
equations (see e.g. \cite{ADN59}). Let $u$ satisfy the following
equation
\begin{equation}
\lap^m u+\sum^{2m-1}_{|\alpha|=0} V_\alpha D^\alpha u=g(x)  \quad
\quad \mbox{in} \ \mathbb B_1, \label{tts}
\end{equation}
where $\|V_\alpha\|_{L^\infty}\leq 1$ for every $\alpha$. Then we
have
\begin{lemma}
 Let
$1<p<\infty$. Suppose $u\in W^{2m, p} $ satisfies (\ref{tts}). Then
there exits a constant $C>0$ depending only on $ n, m $ such that
for any $\sigma\in (0, 1)$,
\begin{equation}
\|u\|_{W^{2m,p}( \mathbb B_{\sigma })}\leq
\frac{C(n,m)}{(1-\sigma)^{2m}}\big(\|g\|_{L^p(\mathbb
B_1)}+\|u\|_{L^p(\mathbb B_1)}\big). \label{apr}
\end{equation}
\label{afa}
\end{lemma}

We are going to establish a $L^\infty$ bound for the solution of
$$  \triangle^m u+ \sum^{\alpha_0}_{|\alpha|=1} V_\alpha(x)\cdot D^\alpha
u+V_0(x) u=0 $$ using the last lemma.

\begin{lemma} Let $u$ be the solution in (\ref{goal}). There exists
a positive constant $C$ independent of $u$ such that
\begin{equation}
\|u\|_{L^\infty(\mathbb B_r)} \leq C
(\sum^{\alpha_0}_{|\alpha|=1}\|V_\alpha\|_{L^\infty}+\|V_0\|_{L^s}+1
)^{\frac{n}{2}} r^{-\frac{n}{2}}\|u\|_{L^2 (\mathbb B_{2r})}
\label{regular}.
\end{equation}
\label{lemma2}
\end{lemma}
\begin{proof}
  We do a scaling argument. Let $w(x)=u(Rrx)$, where $$
R=\frac{1}{C(\sum^{\alpha_0}_{|\alpha|=1}\|V_\alpha\|_{L^\infty}+\|V_0\|_{L^s}^{\frac{s}{2ms-n}}+1)}$$
and $0<r<1$. Then $w(x)$ satisfies the following equation
\begin{equation}
\lap^m w+  \sum^{\alpha_0}_{|\alpha|=1} \tilde{V}_\alpha D^\alpha
w+\tilde{V}_0 w=0, \label{rescaled}
\end{equation}
where
$$\tilde{V}_\alpha=R^{2m-\alpha} r^{2m-\alpha} V_\alpha(Rrx) \quad \mbox{and} \quad \tilde{V}_0=R^{2m}r^{2m} V_0(Rrx).  $$
It is easy to check that $\|\tilde{V}_\alpha\|_{L^\infty}\leq 1$ and
$\| \tilde{V}_0\|_{L^s}\leq 1$. We use the $W^{2m, p}$ estimate to
get the $L^\infty$ bound, then use the scaling argument to find that
how the coefficients depend on the norm of $V_{\alpha}$ and $V_0$.

We first study the case $n>4m-2$ and $s>\frac{2n}{3m}$ for higher
order elliptic equations. Assume that $g(x)$ is $\tilde{V}_0w$ in
Lemma \ref{afa}. By the $W^{2m,p}$ estimate in Lemma \ref{afa}
with $p=\frac{2s}{s+2}$ and H\"older's inequality with
$\frac{s+2}{2s}=\frac{1}{s}+\frac{1}{2}$, we obtain
\begin{align}
\|w\|_{W^{2m, \frac{2s}{s+2}}(\mathbb B_\sigma)}&\leq
C(\sigma)[\|w\|_{L^\frac{2s}{s+2}(\mathbb B_1)}+ \|\tilde{V}_0
w\|_{L^\frac{2s}{s+2}(\mathbb B_1)}]\nonumber \medskip \\
&\leq C(\sigma)(1+\|\tilde{V}_0\|_{L^s}) \|w\|_{L^2(\mathbb B_1)} \nonumber \medskip \\
&\leq  C(\sigma)\|w\|_{L^2(\mathbb B_1)} \label{sobo3}.
\end{align}
By Sobolev embedding $W^{2m, \frac{2s}{s+2}}\hookrightarrow
L^\frac{2ns}{n(s+2)-4ms}$, it follows that
$$\|w\|_{L^\frac{2ns}{(n-4m)s+2n}(\mathbb B_{\sigma})}\leq C(\sigma) \|w\|_{L^2(\mathbb B_1)}.  $$
Obviously,  $w$ in a smaller Lebesgue
 $L^p$ space with $p=\frac{2ns}{(n-4m)s+2n}>2$.

 Using  the $W^{2m,p}$ estimate in Lemma \ref{afa} again with
$p=\frac{2ns}{(n-4m)s+4n}$ and H\"older's inequality with
$\frac{1}{p}=\frac{1}{s}+\frac{(n-4m)s+2n}{2ns}$, we get
\begin{align}
\|w\|_{W^{2m, \frac{2ns}{(n-4m)s+4n}}(\mathbb B_{\sigma^2})}&\leq
C(\sigma)[\|w\|_{L^\frac{2ns}{(n-4m)s+4n}(\mathbb B_\sigma)}+
\|\tilde{V}_0
w\|_{L^\frac{2ns}{(n-4m)s+4n}(\mathbb B_1)}]\nonumber \medskip \\
&\leq C(\sigma)(1+\|\tilde{V}_0\|_{L^s}) \|w\|_{L^\frac{2ns}{(n-4m)s+2n}(\mathbb B_\sigma)} \nonumber \medskip \\
&\leq  C(\sigma)\|w\|_{L^2(\mathbb B_1)} \label{sobo32}.
\end{align}
By Sobolev embedding, it can be deduced that
\begin{equation}
\|w\|_{W^{2m, \frac{2ns}{(n-4m-4)s+4n}}(\mathbb B_{\sigma^2})}\leq
C(\sigma)\|w\|_{L^2(\mathbb B_1)}.
\end{equation}
 If we
 continue this argument as performed
 before,  we will get that $u$ in a smaller $L^p$ space with a larger exponent $p$.
 Assume that we have obtained
 $$\|w\|_{L^\frac{qs}{s-q}(\mathbb B_{\sigma^{k-1}})}\leq C(\sigma) \|w\|_{L^2(\mathbb B_1)}  $$
for some $q$ very close to $\frac{n}{2m}$ after $k-1$ steps.
Furthermore, the $W^{2m,p}$ estimate with $p=q$ and H\"older's
inequality with $\frac{1}{q}=\frac{1}{s}+\frac{s-q}{sq}$ yields that
\begin{align}
\|w\|_{W^{2m, q}(\mathbb B_{\sigma^{k}})}&\leq
C(\sigma)[\|w\|_{L^q(\mathbb B_{\sigma^{k}})}+ \|\tilde{V}_0
w\|_{L^q(\mathbb B_{\sigma^{k}})}]\nonumber \medskip \\
&\leq C(\sigma)\|w\|_{L^\frac{qs}{s-q}(\mathbb B_{\sigma^{k}})} \nonumber \medskip \\
&\leq  C(\sigma)\|w\|_{L^2(\mathbb B_1)} \label{sobo4}.
\end{align}
Keep in mind that $s> \frac{2n}{3m}$. We choose some $q'$ such that
$\frac{n}{2m}<q'<\frac{2n}{3m}$. Since $q$ is very close to
$\frac{n}{2m}$, by Sobolev embedding $W^{2m, q}\hookrightarrow
L^\frac{nq}{n-2mq}$ with $\frac{nq}{n-2mq}>\frac{q's}{s-q'}$, we can
have that
$$ \|w\|_{L^\frac{q's}{s-q'}(\mathbb B_{\sigma^{k}})}\leq
C(\sigma)\|w\|_{L^2 (\mathbb B_1)}   $$ for some $k$ depending on
$n, m$ and $s$. Using the $W^{2m,p}$ estimate with $p=q'$ and
H\"older's inequality again, we have
\begin{align}
\|w\|_{W^{2m, q'}(\mathbb B_{\sigma^{k+1}})}&\leq
C(\sigma)[\|w\|_{L^{q'}(\mathbb B_{\sigma^{k}})}+ \|\tilde{V}_0
w\|_{L^{q'}(\mathbb B_{\sigma^{k}})}]\nonumber \medskip \\
&\leq C(\sigma)\|w\|_{L^\frac{q's}{s-q'}(\mathbb B_{\sigma^{k}})} \nonumber \medskip \\
&\leq  C(\sigma)\|w\|_{L^2(\mathbb B_1)}
\label{sobo5}.
\end{align}
The Sobolev embedding implies that
\begin{equation}
\|w\|_{L^\infty(\mathbb B_{\sigma^{k+1}})}\leq
C(\sigma)\|w\|_{W^{2m, q'}(\mathbb B_{\sigma^{k+1}})}\leq
C(\sigma)\|w\|_{L^2(\mathbb B_1)}. \label{sobo}
\end{equation}
Let $\sigma^{k+1}=\frac{1}{2}$. Recall that $w(x)=u(Rrx)$, the
latter inequality implies
\begin{equation}
\|u\|_{L^\infty(\mathbb B_{{Rr}/2})}\leq C
(Rr)^{-\frac{n}{2}}\|u\|_{L^2 (\mathbb B_{Rr})}. \label{infe}
\end{equation}

If $ 2\leq n\leq 4m-2$, we can carry out the similar argument to get the
$L^\infty$ bound (\ref{infe}). Actually, it takes fewer iterations.

 For any $\mathbb B_r$, assume that the maximum value for $|u|$
is achieved in $\mathbb B_r$ at $x_0$. That is,
$\|u\|_{L^\infty(\mathbb B_r)}=|u(x_0)|$. Using the inequality
(\ref{infe}) at the ball $\mathbb B_{Rr}(x_0)$ yields that
\begin{align}
\|u\|_{L^\infty(\mathbb B_r)}&\leq C(Rr)^{-\frac{n}{2}} \|u\|_{L^2
(\mathbb B_{Rr}(x_0))} \nonumber \\
&\leq C(Rr)^{-\frac{n}{2}} \|u\|_{L^2 (\mathbb B_{2r})}.
\end{align}
Considering the definition of $R$, we obtain
\begin{equation}
\|u\|_{L^\infty(\mathbb B_r)} \leq C
(\sum^{\alpha_0}_{|\alpha|=1}\|V_\alpha\|_{L^\infty}+\|V_0\|_{L^s}^\frac{s}{2ms-n}+1)
^{\frac{n}{2}} r^{-\frac{n}{2}}\|u\|_{L^2 (\mathbb B_{2r})}.
\label{regular1}
\end{equation}
Since $\frac{s}{2ms-n}\leq 1$ if and only if $s\geq \frac{n}{2m-1}$,
we can check $\frac{s}{2ms-n}\leq 1$ from the assumption of $s$ in
all cases $n>4m-2$,$ n= 4m-2$ and $2\leq n<4m-2$. Thus, the estimate
(\ref{regular}) is achieved.

Therefore, we arrive at (\ref{regular}) in the lemma in those three
cases.
\end{proof}

Using the new $L^2\to L^p$ type Carleman estimate in Theorem
\ref{CarlpqVW} for higher order elliptic equations, we establish a
three-ball inequality that plays an important role in obtaining the
vanishing order. The three-ball inequality is also considered as a
quantitative behavior of the strong unique continuation property.
The argument is motivated by those in \cite{Ken07}.

\begin{lemma}
Let $0 < r_0< r_1< R_1 < R_0$, where $R_0 < 1$ is sufficiently
small. Let $u$ be a solution to \eqref{goal} in $B_{R_0}$. \\
I): In the case of $n> 4m-2$, assume that $s \in \pb{\frac{2n}{3m},
\iny}$,
 then
\begin{align}
&\|u\|_{L^\infty \pr{\mathbb B_{3r_1/4}}} \le
C\beta^{2m+\frac{n}{2}} |\log r_1|^m \brac{\|u\|_{L^\iny(\mathbb
B_{2r_0})}}^{k_0}
\brac{\|u\|_{L^\iny(\mathbb B_{R_1})}}^{1 - k_0} \nonumber \\
&+C \beta^{\frac n 2} \pr{\frac{R_1 }{r_1}}^{\frac n 2}
\exp\brac{C(1+\sum_{|\alpha|=1}^{\alpha_0}\|V_\alpha\|_{L^\infty}^\mu+\|V_0\|_{L^s}^\nu)
\pr{\phi\pr{\frac{R_1}{2}}-\phi(r_0)}} \|u\|_{L^\iny(\mathbb
B_{2r_0})}, \label{three}
\end{align}
where $\disp k_0 =
\frac{\phi(\frac{R_1}{2})-\phi(r_1)}{\phi(\frac{R_1}{2})-\phi(r_0)}$,
$\beta =
C(1+\sum_{|\alpha|=1}^{\alpha_0}\|V_\alpha\|_{L^\infty}+\|V_0\|_{L^s})$,
and $\mu$  and $\nu$  are as given in Theorem \ref{CarlpqVW}.

II): In the case of $n=4m-2$, assume that $s \in
\pb{\frac{4(2m-1)}{3m}, \iny}$,
 then
\begin{align}
&\|u\|_{L^\infty \pr{\mathbb B_{3r_1/4}}} \le
C\beta^{2m+\frac{n}{2}} |\log r_1|^m \brac{\|u\|_{L^\iny(\mathbb
B_{2r_0})}}^{k_0}
\brac{\|u\|_{L^\iny(\mathbb B_{R_1})}}^{1 - k_0} \nonumber \\
&+C \beta^{\frac n 2} \pr{\frac{R_1 }{r_1}}^{\frac n 2}
\exp\brac{C(1+\sum_{|\alpha|=1}^{\alpha_0}\|V_\alpha\|_{L^\infty}^\mu+\|V_0\|_{L^s}^{\tilde{\nu}})
\pr{\phi\pr{\frac{R_1}{2}}-\phi(r_0)}} \|u\|_{L^\iny(\mathbb
B_{2r_0})}, \label{threee}
\end{align}
where $\disp k_0$ and $\beta$ are the same as those in Case I, and
$\mu$ and $\tilde{\nu}$ are as given in Theorem \ref{CarlpqVW}.

III): In the case of $2\leq n< 4m-2$, assume that $s \in
\pb{\frac{4(2m-1)}{3m}, \iny}$
 then
\begin{align}
&\|u\|_{L^\infty \pr{\mathbb B_{3r_1/4}}} \le
C\beta^{2m+\frac{n}{2}} |\log r_1|^m \brac{\|u\|_{L^\iny(\mathbb
B_{2r_0})}}^{k_0}
\brac{\|u\|_{L^\iny(\mathbb B_{R_1})}}^{1 - k_0} \nonumber \\
&+C \beta^{\frac n 2} \pr{\frac{R_1 }{r_1}}^{\frac n 2}
\exp\brac{C(1+\sum_{|\alpha|=1}^{\alpha_0}\|V_\alpha\|_{L^\infty}^\mu+\|V_0\|_{L^s}^{\bar\nu})
\pr{\phi\pr{\frac{R_1}{2}}-\phi(r_0)}} \|u\|_{L^\iny(\mathbb
B_{2r_0})}, \label{threeee}
\end{align}
where $\disp k_0$ and $\beta$ are the same as those in Case I and
II, and $\mu$ and $\bar\nu$ are as given in Theorem \ref{CarlpqVW}.

\end{lemma}

\begin{proof}
We first consider the case $n> 4m-2$ and $s \in \pb{\frac{2n}{3m},
\iny}$. Let $r_0< r_1< R_1$. The standard notation $\brac{a,b}$ is
denoted as  closed annulus with inner radius $a$ and outer radius
$b$. Choose a smooth function $\eta\in C^\infty_{0}(\mathbb
B_{R_0})$ with $B_{2R_1}\subset B_{R_0}$. Let
$$D_1=\brac{\frac{3}{2}r_0, \frac{1}{2}R_1 }, \quad  \quad
D_2= \brac{r_0, \frac{3}{2}r_0}, \quad \quad
D_3=\brac{\frac{1}{2}R_1, \frac{3 }{4}R_1}.$$ We define $\eta$ as
$\eta=1$ on $D_1$ and $\eta=0$ on $[0, \ r_0]\cup
\brac{\frac{3}{4}R_1, \ R_1}$. Then we have $| D^\alpha \eta|\leq
\frac{C}{r_0^{|\alpha|}}$ on $D_2$. Similarly, $| D^\alpha \eta|\leq
\frac{C}{R_1^{|\alpha|}}$ on $D_3$.

Since $u$ is a solution to \eqref{goal} in $\mathbb B_{R_0}$, by
regularity argument mentioned in the introduction, $u \in
L^\iny\pr{\mathbb B_{R_1}} \cap W^{2m,2}\pr{\mathbb B_{R_1}}$.
Therefore, by regularization, the estimate in Theorem \ref{CarlpqVW}
holds for $\eta u$. To use the Carleman estimates in Theorem
\ref{CarlpqVW}, we
 substitute $\eta u$ into (\ref{main}). The following holds
\begin{align}
&\tau^{\be_0} \|e^{-\tau \phi(r)}(\log r)^{-m} \eta
u\|_{L^p(r^{-n}dx)} + \sum^{2m-1}_{|\alpha|=1} \tau^{\be_\alpha}
\|e^{-\tau \phi(r)}(\log r )^{-m} r^{|\alpha|}
 D^\alpha u\|_{L^2(r^{-n}dx)}
\nonumber \medskip\\
&\leq  C \| e^{-\tau \phi(r)} r^{2m} (\triangle^m (\eta u)+
\sum_{|\alpha|=1}^{\alpha_0} V_\alpha D^\alpha (\eta u)+ V_0 \eta
u)\|_{L^2(r^{-n} dx)} ,
\end{align}
whenever
$$  \tau>C(1+\sum_{|\alpha|=1}^{\alpha_0}\|V_\alpha\|_{L^\infty}^\mu+\|V_0\|_{L^p}^\nu).       $$
Consider that $u$ is a solution to equation \eqref{goal}, further
calculations show that
\begin{align}
&\tau^{\be_0} \|e^{-\tau \phi(r)}(\log r)^{-m} \eta
u\|_{L^p(r^{-n}dx)} \nonumber \medskip\\&\leq  C \| e^{-\tau
\phi(r)} r^{2m} ([\triangle^m, \eta]u+\eta\triangle^m u+
\sum_{|\alpha|=1}^{\alpha_0} V_\alpha [D^\alpha,  \eta]
u+\sum_{|\alpha|=1}^{\alpha_0} \eta D^\alpha  u   + V_0 \eta
u)\|_{L^2(r^{-n} dx)}
 \nonumber \medskip\\ &\leq  C \| e^{-\tau \phi(r)} r^{2m}
([\triangle^m, \eta]u+\sum_{|\alpha|=1}^{\alpha_0} V_\alpha
[D^\alpha,  \eta] u\|_{ L^2 (r^{-n} dx)}.
\end{align}
Note that $[\triangle^m, \eta]$ is a $2m-1$ order differential
operator on $u$ involving the derivative of $\eta$. From the last
inequality and $p\geq 2$, we have
\begin{equation}
\tau^{\be_0} \|e^{-\tau \phi(r)}(\log r)^{-m} \eta
u\|_{L^2(r^{-n}dx)}\leq C \mathcal{K}, \label{jaja}
\end{equation}
where \begin{align}\mathcal{K}&=\| e^{-\tau \phi(r)} r^{2m}
([\triangle^m, \eta]u+\sum_{|\alpha|=1}^{\alpha_0} V_\alpha
[D^\alpha,  \eta] u\|_{ L^2 (r^{-n} dx)}\nonumber \\
&\leq \sum^{2m-1}_{\alpha=0} C\|e^{-\tau \phi(r)}
r^{|\alpha|}|D^\alpha u|\|_{ L^2 ( D_2\cup D_3, r^{-n} dx)}
\nonumber \\
&+\sum_{|\alpha|=1}^{\alpha_0} (\|V_\alpha\|_{L^\infty}+1) \|
e^{-\tau \phi(r)} r^{|\alpha|}|D^\alpha  u|\|_{ L^2 ( D_2\cup D_3,
r^{-n} dx)}.
\end{align}
Recall that $\beta=C(\sum_{|\alpha|=1}^{
\alpha_0}\|V_{\alpha}\|_{L^\infty}+\|V_0\|_{L^s}+1).$ By Lemma
\ref{hormmm} and the fact that $-\phi(r)$ is decreasing, we have
\begin{align}
 \| e^{-\tau \phi(r)}
r^{|\alpha|}D^\alpha  u\|_{ L^2 ( D_2, r^{-n} dx)} &\leq C e^{-\tau
\phi(r_0)}r_0^{-\frac{n}{2}}  \| r^{|\alpha|}D^\alpha u\|_{ L^2 (
D_2, dx)} \nonumber \\
& \leq C\beta^{2m-1} e^{-\tau \phi(r_0)}r_0^{-\frac{n}{2}}
\|u\|_{L^2\pr{\mathbb B_{2r_0}\backslash \mathbb B_{{r_0}/{2}}}}.
\end{align}
Similarly,
\begin{align}
 \| e^{-\tau \phi(r)}
r^{|\alpha|}D^\alpha  u\|_{ L^2 ( D_3, r^{-n} dx)} &\leq e^{-\tau
\phi(\frac{R_1}{2})}R_1^{-\frac{n}{2}}  \| r^{|\alpha|}D^\alpha
u\|_{ L^2 (
D_3, dx)} \nonumber \\
& \leq C\beta^{2m-1} e^{-\tau \phi(\frac{R_1}{2})}R_1^{-\frac{n}{2}}
\|u\|_{L^2\pr{\mathbb B_{R_1}\backslash \mathbb B_{{R_1}/{4}}}}.
\end{align}
We conclude that
\begin{eqnarray*}
\mathcal{K} &\leq & C\beta^{2m} e^{-\tau
\phi(r_0)}r_0^{-\frac{n}{2}} \|u\|_{L^2\pr{\mathbb
B_{2r_0}\backslash
\mathbb B_{{r_0}/{2}}}}\nonumber \\
 &+& C\beta^{2m} e^{-\tau
\phi(\frac{R_1}{2})}R_1^{-\frac{n}{2}} \|u\|_{L^2\pr{\mathbb
B_{R_1}\backslash \mathbb B_{{R_1}/{4}}}}.
\end{eqnarray*}
Define a new set $D_4=\{r\in D_1, \ r\leq r_1\}$. From \eqref{jaja}
and the fact that $\tau \ge 1$ and $\be_0 > 0$, it is attained that
\begin{align*}
\| u\|_{L^2 (D_4)} &\le \tau^{\beta_0}\| u\|_{L^2 (D_4)} \\ &\le
\tau^{\beta_0} \|e^{\tau\phi(r)}(\log r)^m
r^{\frac{n}{2}}\|_{L^\iny\pr{D_4}}
 \|e^{-\tau\phi(r)} (\log r)^{-m} u\|_{L^2 (D_4, r^{-n}dx)}  \\
&\le e^{\tau \phi(r_1)} |\log r_1|^m r_1^{\frac{n}{2}} \mathcal{K},
\end{align*}
where the fact that $e^{\tau\phi(r)}|\log r|^m r^{\frac{n}{2}}$ is
increasing on $D_4$ for $R_0$ sufficiently small is used. Adding
$\|u\|_{L^2 \pr{\mathbb B_{3r_0/2}}}$ to both sides of the last
inequality and taking the upper bound of $\mathcal{K}$ into account
yields that
\begin{align*}
\| u\|_{L^2 (\mathbb B_{r_1})} &\le C |\log r_1|^m
\beta^{2m}\pr{\frac{r_1}{r_0}}^{\frac{n}{2}}e^{\tau \brac{\phi(r_1)-
\phi(r_0)}}\|u\|_{L^2(\mathbb B_{2r_0})} \\
&+ C |\log r_1|^m \beta^{2m}
\pr{\frac{r_1}{R_1}}^{\frac{n}{2}}e^{\tau\brac{\phi(r_1) -
\phi\pr{\frac{R_1}{2}}}} \|u\|_{L^2(\mathbb B_{R_1})}.
\end{align*}
Let $U_1 =\|u\|_{L^2(\mathbb B_{2r_0})}$, $U_2=\|u\|_{L^2(\mathbb
B_{R_1})}$ and define
\begin{align*}
B_1 &= C |\log r_1|^m   \beta^{2m} \pr{\frac{r_1}{r_0}}^{\frac{n}{2}},  \\
B_2 &= C |\log r_1|^m \beta^{2m}\pr{\frac{r_1}{R_1}}^{\frac{n}{2}}.
\end{align*}
Then the last inequality leads to
\begin{eqnarray}
\| u\|_{L^2 (\mathbb B_{r_1})}  &\leq& B_1 \brac{\frac{\exp
\pr{\phi(r_1)}}{\exp\pr{ \phi(r_0)}}}^\tau U_1 + B_2
\brac{\frac{\exp\pr{ \phi(r_1)}}{\exp\pr{
\phi\pr{\frac{R_1}{2}}}}}^\tau U_2. \label{D4est}
\end{eqnarray}
Define a new parameter $k_0$ as follow
$$\frac{1}{k_0}=\frac{\phi(\frac{R_1}{2})-\phi(r_0)}{\phi(\frac{R_1}{2})-\phi(r_1)}.$$
Recall that $\phi(r)=\log r+\log (\log r)^2$. If we fix $r_1$ and
$R_1$, and choose $r_0$ to be sufficiently small, i.e. $r_0\ll r_1$,
then $\frac{1}{k_0}\simeq \log \frac{1}{r_0}$. Set
$$\tau_1 =\frac{k_0}{\phi\pr{\frac{R_1}{2}}-\phi(r_1)}\log\pr{\frac{B_2{U}_2}{B_1 {U}_1}}.$$
If $\tau_1 \ge C(1+\sum_{|\alpha|=1}^{
\alpha_0}\|V_\alpha\|_{L^\infty}^\mu+\|V_0\|_{L^s}^\nu)$, then the
previous calculations hold with $\tau = \tau_1$. We get from
\eqref{D4est} that
\begin{eqnarray}
\| u\|_{L^2 (\mathbb B_{r_1})}  &\leq& 2\pr{B_1 U_1}^{k_0}\pr{B_2
U_2}^{1 - k_0}. \label{mix1}
\end{eqnarray}
On the other hand, if $\tau_1
<C(1+\sum_{\alpha=1}^{\alpha_0}\|V_\alpha\|_{L^\infty}^\mu+\|V_0\|_{L^s}^\nu)
$, it follows that
\begin{align*}
U_2 < \frac{B_1}{B_2}
\exp\brac{C(1+\sum_{|\alpha|=1}^{\alpha_0}\|V_\alpha\|_{L^\infty}^\mu+\|V_0\|_{L^s}^\nu)
\pr{\phi\pr{\frac{R_1}{2}}-\phi(r_0)}} U_1.
\end{align*}
We can write the last inequality as
\begin{equation}
\|u\|_{L^2 (\mathbb B_{r_1})} \le C \pr{\frac{R_1}{r_0}}^{\frac n 2}
e^{C(1+\sum_{|\alpha|=1}^{
\alpha_0}\|V_\alpha\|_{L^\infty}^\mu+\|V_0\|_{L^p}^\nu)
\pr{\phi\pr{\frac{R_1}{2}}-\phi(r_0)}} \|u\|_{L^2(\mathbb
B_{2r_0})}. \label{mix2}
\end{equation}
Together with \eqref{mix1} and \eqref{mix2}, we obtain that
\begin{align}
\| u\|_{L^2 (\mathbb B_{r_1})} &\le C ( \sum_{|\alpha|=1}^{
\alpha_0}\|V_{\alpha}\|_{L^\infty}+\|V_0\|_{L^s}+1)^{2m}  |\log r_1|^m r_1^{\frac n 2}\brac{ r_0^{-\frac{n}{2}}
\|u\|_{L^2(\mathbb B_{2r_0})}}^{k_0} \nonumber \\
&\times \brac{R_1^{-\frac{n}{2}} \|u\|_{L^2(\mathbb B_{R_1})}}^{1 - k_0} \nonumber \\
&+C  \pr{\frac{R_1}{r_0}}^{\frac n 2} e^{C(1+\sum_{|\alpha|=1}^{
\alpha_0}\|V_\alpha\|_{L^\infty}^\mu+\|V_0\|_{L^s}^\nu)
\pr{\phi\pr{\frac{R_1}{2}}-\phi(r_0)}} \|u\|_{L^2(\mathbb
B_{2r_0})}.
 \label{end2}
\end{align}
Recall from Lemma \ref{lemma2} that
\begin{align}
\|u\|_{L^\infty(\mathbb B_r)} \leq C (\sum_{|\alpha|=1}^{
\alpha_0}\|V_\alpha\|_{L^\infty}+\|V_0\|_{L^s}+1)^{\frac{n}{2}}
r^{-\frac{n}{2}}\|u\|_{L^2 (\mathbb B_{2r})}. \label{ell}
\end{align}
Combining the estimates \eqref{end2} and \eqref{ell}, the three-ball
inequality in the $L^\infty$-norm in the form of \eqref{three} is
derived.

For the case $n=4m-2$ and $s \in \pb{\frac{4(2m-1)}{3m}, \iny}$, the
same argument using the Carleman estimates (\ref{main1}) will give
\eqref{threee}. If $2\leq n <4m-2$ and $s \in \pb{\frac{4(2m-1)}{3m}, \iny}$, we can also
obtain the inequality (\ref{threeee}) from the Carleman estimates
(\ref{main3}) by performing the same argument as Case I. This
completes the proof of the lemma.
\end{proof}

The inequalities \eqref{three}, \eqref{threee} and \eqref{threeee}
are the three-ball inequalities we use in the proof of Theorem
\ref{thh}. We first use the three-ball inequality in the propagation
of smallness argument to establish a lower bound for the solution on
$\mathbb B_r$. Similar arguments have been performed in
\cite{Zhu16}. Then we use the three-ball inequality again to
establish the order of vanishing estimate.

\begin{proof} [Proof of Theorem  \ref{thh}]
Without loss of generality, we may assume that $x_0$ is the origin.
We first consider the case $I$. Let $r_0=\frac{r}{2}$, $r_1=4r$ and
$R_1=10r$. Then the estimate \eqref{three} implies that
\begin{align}
\|u\|_{L^\infty \pr{\mathbb B_{3r}}} &\le  C \beta^{\frac{4m+n}{2}}  |\log r|^m
\|u\|_{L^\iny(\mathbb B_{r})}^{k_0} \|u\|_{L^\iny(\mathbb B_{10r})}^{1 - k_0} \nonumber \\
&+C \exp\brac{ C(1+\sum_{|\alpha|=1}^{
\alpha_0}\|V_\alpha\|_{L^\infty}^\mu+\|V_0\|_{L^s}^\nu)
\pr{\phi\pr{5r}-\phi\pr{\frac r 2}} } \|u\|_{L^\iny(\mathbb B_{r})},
\label{refi}
\end{align}
where $\disp k_0 = \frac{\phi(5r)-\phi(4r)}{\phi(5r)-\phi\pr{\frac r
2}}$. It is obvious that
$$c\leq \phi(5r)-\phi\pr{\frac{r}{2}}\leq C \quad \mbox{and} \quad c\leq \phi(5r)-\phi(4r)\leq C,$$
where $C$ and $c$ are positive constants are independent of $r$.
Thus, the parameter $k_0$ does not depend on  $r$.

We choose a small $r < \frac 1 2$ such that
$$\sup_{\mathbb B_r(0)}|u|=\delta,$$
where $\delta>0$. Otherwise, by the unique continuation, $u\equiv 0$
in $\mathbb B_1$, which is impossible. Since $\disp \sup_{|x|\leq
1}|u(x)|\geq 1$, by continuity, there exists some $\bar x\in \mathbb
B_1$ such that $\disp \abs{u(\bar x)}=\sup_{|x|\leq 1}|u(x)|\geq 1$.
There also exists a sequence of balls with radius $r$, centered at
$x_0=0, \ x_1, \ldots, x_d$ so that $x_{i+1}\in \mathbb B_{r}(x_i)$
for every $i$, and $\bar x\in \mathbb B_{r}(x_d)$. The number of
balls, $d$, depends on the radius $r$ that will be fixed later. The
application of $L^\infty$-version of three-ball inequality
(\ref{refi}) at the origin and the boundedness assumption that
$\|u\|_{L^\infty(\mathbb B_{10})}\leq \hat{C}$ yield that
\begin{align*}
\|u\|_{L^\infty \pr{\mathbb B_{3r}(0)}} &\le C \delta^{k_0} (
\sum_{1\leq \alpha\leq
2m-1}\|V_{\alpha}\|_{L^\infty}+\|V_0\|_{L^s}+1)^{C}  |\log r|^m  \nonumber \\
&+ \delta \exp\brac{ C(1+\sum_{|\alpha|=1}^{
\alpha_0}\|V_\alpha\|_{L^\infty}^\mu+\|V_0\|_{L^s}^\nu)}.
\end{align*}
By the way each $\mathbb B_r(x_i)$ is chosen, we obtain $\mathbb
B_r(x_{i+1})\subset \mathbb B_{3r}(x_{i})$. Hence, for every $i = 1,
2, \ldots, d$,
\begin{equation}
\|u\|_{L^\infty (\mathbb B_r(x_{i+1}))}\leq  \|u\|_{L^\infty
(\mathbb B_{3r}(x_{i}))}. \label{bbb}
\end{equation}
Repeating the above argument with balls centered at $x_i$ and making
use of \eqref{bbb} give that
\begin{equation*}
\|u\|_{L^\infty (\mathbb B_{3r}(x_{i}))} \leq C_i \delta^{D_i} |\log
r|^{F_i} \exp\brac{H_i(1+\sum_{|\alpha|=1}^{
\alpha_0}\|V_\alpha\|_{L^\infty}^\mu+\|V_0\|_{L^s}^\nu) }
\end{equation*}
for $i=0, 1, \cdots, d$, where $C_i$ is a constant depending on $d$,
$n$, $m$, $s$, $\hat C$, and $C$ from Lemma \ref{CarlpqVW} and
$D_i$, $E_i$, $F_i$ $H_i$ are constants depending on $n$, $m$, and
$d$. By the fact that $\abs{u(\bar x)} \geq 1$ and $\bar x \in
B_{3r}(x_d)$, we get
\begin{equation*}
\delta \ge c \exp\brac{- C(1+\sum_{|\alpha|=1}^{
\alpha_0}\|V_\alpha\|_{L^\infty}^\mu+\|V_0\|_{L^s}^\nu) } |\log
r|^{-C},
\end{equation*}
where $C$ depends on $d$, $n$, $m$, and $\hat C$.

Now the radius $r$ is fixed as a small number so that $d$ is a fixed
constant. We are going to apply the three-ball inequality again. Let
$\frac{3}{4}r_1=r$, $R_1=10r$ and let $r_0 << r$, i.e. $r_0$ is
sufficiently small with respect to $r$. Hence, the three-ball
inequality \eqref{three} implies that
$$ \delta \leq {I}_1 +I_2,$$
where
\begin{align*}
{ I}_1 &= C \beta^{2m+\frac n 2}  |\log r|^m \brac{
\|u\|_{L^\iny(\mathbb B_{2r_0})}}^{k_0} \brac{  \|u\|_{L^\iny(\mathbb B_{10 r})}}^{1 - k_0}, \\
I_2 &= C \beta^{\frac n 2} e^{C(1+\sum_{|\alpha|=1}^{
\alpha_0}\|V_\alpha\|_{L^\infty}^\mu+\|V_0\|_{L^s}^\nu)
\pr{\phi\pr{5r}-\phi(r_0)}} \|u\|_{L^\iny(\mathbb B_{2r_0})}
\end{align*}
with $\disp k_0 = \frac{\phi(5r)-\phi(\frac 4 3
r)}{\phi(5r)-\phi(r_0)}$.

If ${I}_1 \leq I_2$, we have
\begin{align*}
&\exp\brac{- C(1+\sum_{|\alpha|=1}^{
\alpha_0}\|V_\alpha\|_{L^\infty}^\mu+\|V_0\|_{L^s}^\nu) } |\log
r|^{-C}
\le \delta \le 2 I_2 \\
&\le 2 \beta^{\frac n 2}e^{C(1+\sum_{|\alpha|=1}^{
\alpha_0}\|V_\alpha\|_{L^\infty}^\mu+\|V_0\|_{L^s}^\nu)
\pr{\phi\pr{5r}-\phi(r_0)}} \|u\|_{L^\iny(\mathbb B_{2r_0})}.
\end{align*}
Since $r_0 << r$, it is true that $\phi\pr{r_0}-\pr{C + \phi\pr{5r}}
\ge c \phi\pr{r_0}$. We get that
\begin{align*}
\|u\|_{L^\iny(\mathbb B_{2r_0})} &\ge C r_0^{C(1+\sum_{|\alpha|=1}^{
\alpha_0}\|V_\alpha\|_{L^\infty}^\mu+\|V_0\|_{L^s}^\nu)}.
\end{align*}
Instead, if $I_2 \leq { I}_1$, we obtain that
\begin{align*}
&\exp\brac{- C(1+\sum_{|\alpha|=1}^{
\alpha_0}\|V_\alpha\|_{L^\infty}^\mu+\|V_0\|_{L^s}^\nu) } |\log
r|^{-C}
\le \delta \le 2 {I}_1 \\
&\le 2 C \beta^{2m+\frac n 2} |\log r|^m \brac{
\|u\|_{L^\iny(\mathbb B_{2r_0})}}^{k_0} \brac{ \|u\|_{L^\iny(\mathbb
B_{10 r})}}^{1 - k_0}.
\end{align*}
If we raise both sides to $\frac{1}{k_0}$ in the last inequality and
take the assumption $\|u\|_{L^\infty\pr{B_{10r}}}\leq \hat{C}$ into
consideration, it follows that
\begin{align*}
 \|u\|_{L^\iny(\mathbb B_{2r_0})}
 &\ge C \pr{\frac{C}{\hat C \abs{\log r}^C }}^{\frac 1 {k_0}} \exp\brac{-\frac{C}{k_0}
(1+\sum_{|\alpha|=1}^{
\alpha_0}\|V_\alpha\|_{L^\infty}^\mu+\|V_0\|_{L^s}^\nu)}.
\end{align*}
Recall that $\frac{1}{k_0}\simeq\log \frac{1}{r_0}$ if $r_0$ is
sufficiently small compared with $r$. We arrive at
\begin{align*}
 \|u\|_{L^\iny(\mathbb B_{2r_0})}
 &\ge C r_0^{C(1+\sum_{|\alpha|=1}^{
\alpha_0}\|V_\alpha\|_{L^\infty}^\mu+\|V_0\|_{L^s}^\nu)
 },
\end{align*}
which shows the proof of case I in Theorem \ref{thh}.

The case $\Pi$ or III follows from the same argument using the
three-ball inequality (\ref{threee}) or (\ref{threeee}). Therefore,
the proof of Theorem \ref{thh} is done.
\end{proof}

\section{Quantitative unique continuation at infinity}
\label{QuantUC}

In this section, we show the proof Theorem \ref{UCVW}. By the
maximal order of vanishing estimates, the quantitative unique
continuation at infinity is established using the idea of scaling
arguments in \cite{BK05}.

\begin{proof}[Proof of Theorem \ref{UCVW}]
Case I): We consider the case $n> 4m-2$ and $s \in
\pb{\frac{2n}{3m}, \iny}$. Assume $u$ be a solution to \eqref{goal}
in $\R^n$. Let $x_0 \in \R^n$ and set $\abs{x_0} = R$.
Define $u_R(x) = u(x_0 + Rx)$. 
Set $$V_{\alpha, R}\pr{x} = R^{2m-|\alpha|} \, V_\alpha\pr{x_0 + R
x} \quad \mbox{and} \quad V_{0, R}\pr{x} = R^{2m} V_0\pr{x_0 + R
x}.$$ For any $r
> 0$, elementary calculations show that
\begin{align*}
\norm{V_{0, R}}_{L^s\pr{\mathbb B_r\pr{0}}} &=
\pr{\int_{\mathbb B_r\pr{0}} \abs{R^{2m} \, V_0\pr{x_0 + R x}}^s dx}^{\frac 1 s} \\
&=  R^{2m - \frac n s } \norm{V_0}_{L^s\pr{\mathbb B_{r
R}\pr{x_0}}}.
\end{align*}
Thus,
$$\disp \|V_{0,
R}\|_{L^s\pr{\mathbb B_{10}\pr{0}}} \le A_0 R^{2m - \frac n s}.$$ It
is clear that
$$\disp \|V_{\alpha, R}\|_{L^\infty\pr{\mathbb B_{10}\pr{0}}}
= R^{2m-\alpha} \|V_\alpha\|_{L^\infty\pr{\mathbb B_{10R}\pr{x_0}}}
\le A_{\alpha} R^{2m-\alpha}.$$ We can check that $u_R$ satisfies
the following scaled version of \eqref{goal} in $\mathbb B_{10}$,
\begin{align}
& \LP u_R\pr{x} + \sum_{|\alpha|=1}^{
\alpha_0} V_{\alpha, R}\pr{x}D^\alpha u_R\pr{x}  + V_{0,R}\pr{x} u_R\pr{x} \nonumber \\
&= R^{2m} \LP u\pr{x_0 + R x} + \sum_{|\alpha|=1}^{ \alpha_0}
R^{2m-\alpha}V_{\alpha, R}\pr{x} R^\alpha D^\alpha u_R\pr{x}
\nonumber \\& +
R^{2m} V_{0, R}\pr{x_0 + R x}u\pr{x_0 + R x} \nonumber \\
&= 0.\label{rescal}
\end{align}
 Obviously,
\begin{align*}
\norm{u_R}_{L^\iny\pr{\mathbb B_{10}}} &=
\norm{u}_{L^\iny\pr{\mathbb B_{10R}\pr{x_0}}} \le C_0.
\end{align*}
Set $\disp\widetilde{x_0} := -x_0/R$. Then $\disp| \widetilde{x_0}|
= 1$ and $\abs{u_R(\widetilde{x_0})} = \abs{u(0)} \ge 1$. Namely,
$\disp\norm{u_R}_{L^\iny(B_1)} \ge 1$. Therefore, if $R >> 1$, then
the application of Theorem \ref{thh} to $u_R$
 and
$\hat C = C_0$ yields that
\begin{align*}
\norm{u}_{L^\iny\pr{{\mathbb B_{1}(x_0)}}} = & \norm{u_R}_{L^\iny\pr{\mathbb B_{1/R}(0)}}  \\
\ge & c(1/R)^{^{C\brac{1 + \sum_{|\alpha|=1}^{ \alpha_0}\pr{A_\alpha
R^{2m - |\alpha|}}^\mu + C_2 \pr{A_0 R^{2m - \frac n s} }^\nu}}} \\
= &c \exp\set{-C\brac{1 + \sum_{|\alpha|=1}^{ \alpha_0} \pr{A_\alpha
R^{2m - |\alpha|}}^\mu + C_2 \pr{A_0 R^{2m - \frac n s} }^\nu} \log
R}.
\end{align*}
Recall that $\mu=\frac{2}{3m-2|\alpha|}$ and
$\nu=\frac{2s}{3ms-2n}$. We can check that $(2m-|\alpha|)\mu$ is
increasing with respect to $|\alpha|$. So its maximum value is
achieved at $\frac{2(2m-\alpha_0)}{3m-\alpha_0}$. It can be shown
that

\begin{align*}
\max\set{(2m-|\alpha|)\mu, \ \pr{2m - \frac n s}\nu }= \Theta :=
\left\{\begin{array}{ll}
\frac{2(2m -\alpha_0)}{3m-2\alpha_0} \quad &  \alpha_0 \geq \frac{n}{s}, \bigskip \\
\frac{2(2ms-n)}{3ms-2n} & \alpha_0 < \frac{n}{s}.
\end{array}\right.
\end{align*}
Therefore,
\begin{align*}
\norm{u}_{L^\iny\pr{{\mathbb B_{1}(x_0)}}} \ge &c \exp\brac{-C\pr{n,
m, s, A_0, \cdots, A_{\alpha_0}} R^\Theta \log R}.
\end{align*}

Case II): In the case of $n=4m-2$ and $s\in (\frac{4(2m-1)}{3m}, \,
\infty]$, similar arguments work. We need to find the
$\max\set{(2m-|\alpha|)\mu, \ \pr{2m - \frac n s}\tilde{\nu} }$.
Recall that $$\tilde{\nu}=\frac{2s}{3ms-4(2m-1)-2(2m-1)(s-2)\eps}.$$
We can check that
\begin{align*}
&\max\set{(2m-|\alpha|)\mu, \ \pr{2m - \frac n s}\tilde{\nu} }
\nonumber \\ &= \tilde{\Theta} := \left\{\begin{array}{ll}
\frac{2(2m -\alpha_0)}{3m-2\alpha_0} &  \alpha_0 \geq \frac{8m(2m-1)-3mn+4m(2m-1)(s-2)\epsilon}{ms+4(2m-1)-2n+2(2m-1)(s-2)\epsilon}, \bigskip \\
\frac{2(2ms-n)}{3ms-4(2m-1)-2(2m-1)(s-2)\epsilon} & \alpha_0 <
\frac{8m(2m-1)-3mn+4m(2m-1)(s-2)\epsilon}{ms+4(2m-1)-2n+2(2m-1)(s-2)\epsilon}.
\end{array}\right.
\end{align*}
Then
\begin{align*}
\norm{u}_{L^\iny\pr{{\mathbb B_{1}(x_0)}}} \ge c \exp\brac{-C\pr{n,
m, s, \eps, A_0, \cdots, A_{{\alpha_0}}} R^{\tilde{\Theta}} \log R}.
\end{align*}

Case III): For the case $2\leq n<4m-2$ and $s\in (\frac{4(2m-1)}{3m}, \,
\infty]$. As before, we need to find the
$\max\set{(2m-|\alpha|)\mu, \ \pr{2m - \frac n s}\bar{\nu} }$.
We can check that
\begin{align*}
&\max\set{(2m-|\alpha|)\mu, \ \pr{2m - \frac n s}\bar{\nu} }
\nonumber \\ &= \bar{\Theta} := \left\{\begin{array}{ll}
\frac{2(2m -\alpha_0)}{3m-2\alpha_0} &  \alpha_0 \geq \frac{8m(2m-1)-3mn}{ms+4(2m-1)-2n}, \bigskip \\
\frac{2(2ms-n)}{3ms-4(2m-1)} & \alpha_0 <
\frac{8m(2m-1)-3mn}{ms+4(2m-1)-2n}.
\end{array}\right.
\end{align*}
Thus,
$$\norm{u}_{L^\iny\pr{{\mathbb B_{1}(x_0)}}} \ge c \exp\brac{-C\pr{n,
m, A_0, \cdots, A_{{\alpha_0}}} R^{\bar{\Theta}} \log R}.
$$

 Therefore, the conclusion of the theorem follows.
\end{proof}

\end{document}